\documentclass[11pt,reqno]{amsart}

\usepackage{a4wide}
\usepackage{amsmath,amssymb}
\usepackage{bbm}
\usepackage[dvipsnames]{xcolor}
\usepackage{latexsym}
\usepackage{tikz-cd}
\allowdisplaybreaks

\usepackage{scalerel,stackengine}
\stackMath

\newcommand\reallywidehat[1]{%
\savestack{\tmpbox}{\stretchto{%
  \scaleto{%
    \scalerel*[\widthof{\ensuremath{#1}}]{\kern-.6pt\bigwedge\kern-.6pt}%
    {\rule[-\textheight/2]{1ex}{\textheight}}
  }{\textheight}%
}{0.5ex}}%
\stackon[1pt]{#1}{\tmpbox}%
}

\newcommand\reallywidecheck[1]{%
\savestack{\tmpbox}{\stretchto{%
  \scaleto{%
    \scalerel*[\widthof{\ensuremath{#1}}]{\kern-.6pt\bigwedge\kern-.6pt}%
    {\rule[-\textheight/2]{1ex}{\textheight}}
  }{\textheight}%
}{0.5ex}}%
\stackon[1pt]{#1}{\scalebox{-1}{\tmpbox}}%
}

\numberwithin{equation}{section}

\newcommand{\Z}{{\mathbb Z}}

\newcommand{\R}{{\mathbb R}}
\newcommand{\N}{{\mathbb N}}
\newcommand{\C}{{\mathbb C}}

\newcommand{\XX}{{\mathbb X}}

\newcommand{\cA}{{\mathcal A}}

\newcommand{\cD}{{\mathcal D}}
\newcommand{\im}{{\mathrm{i}}}
\newcommand{\supp}{{\operatorname{supp}}}

\newcommand{\e}{\operatorname{e}}

\newcommand{\dd}{\mbox{d}}

\newcommand{\cM}{{\mathcal M}}

\newcommand{\cF}{{\mathcal F}}

\newcommand{\cS}{{\textsf S}}
\newcommand{\cSM}{\mathcal{SM}}
\newcommand{\WAP}{\mathcal{W}\hspace*{-1pt}\mathcal{AP}}
\newcommand{\SAP}{\mathcal{S}\hspace*{-2pt}\mathcal{AP}}

\newcommand{\lm}{\ensuremath{\lambda\!\!\!\lambda}}

\newcommand{\sWAP}{\texttt{WAP}^{sm}}
\newcommand{\sSAP}{\texttt{SAP}^{sm}}

\newcommand{\Cu}{C_{\mathsf{u}}}
\newcommand{\Cc}{C_{\mathsf{c}}}
\newcommand{\Cz}{C^{}_{0}}

\newcommand{\SM}{\mathcal{SM}}

\theoremstyle{plain}
\newtheorem{theorem}{Theorem}[section]
\newtheorem{proposition}[theorem]{Proposition}
\newtheorem{lemma}[theorem]{Lemma}
\newtheorem{corollary}[theorem]{Corollary}
\newtheorem{fact}[theorem]{Fact}
\newtheorem{nota}[theorem]{Notation}
\theoremstyle{definition}
\newtheorem{definition}[theorem]{Definition}
\newtheorem{remark}[theorem]{Remark}
\newtheorem{example}[theorem]{Example}
\newtheorem{question}[theorem]{Question}

\begin{document}
\title{Semi-measures and their Fourier transform}

\author{Timo Spindeler}
\address{Fakult\"at f\"ur Mathematik, Universit\"at Bielefeld, \newline
\hspace*{\parindent} Postfach 100131, 33501 Bielefeld, Germany}
\email{tspindel@math.uni-bielefeld.de}

\author{Nicolae Strungaru}
\address{Department of Mathematical Sciences, MacEwan University \newline
\hspace*{\parindent} 10700 -- 104 Avenue, Edmonton, AB, T5J 4S2, Canada\\
and \\
Institute of Mathematics ``Simon Stoilow''\newline
\hspace*{\parindent}Bucharest, Romania}
\email{strungarun@macewan.ca}
\urladdr{http://academic.macewan.ca/strungarun/}

\begin{abstract}
The basic theory of semi-measures on locally compact Abelian groups is extended to prove the existence of a generalised Eberlein decomposition into such semi-measures.
\end{abstract}



\maketitle

\section{Introduction}

The Fourier transform plays a central role in mathematical diffraction theory. Typically, diffraction is defined by starting with a point set $\Lambda \subseteq \widehat{G}$, or more generally a measure $\omega$, constructing its autocorrelation (or 2-point correlation) measure $\gamma$ and taking the Fourier transform $\widehat{\gamma}$. This is a positive measure on the dual group $\widehat{G}$, which models the physical process of diffraction.

The discovery of quasi-crystals in the 1980's \cite{She} shattered many assumptions physicists made, and emphasized the need for a better understanding of diffraction, and its spectral components. It also became clear that we need a better understanding of pure point diffraction, and much progress has been made in this direction. Under extra assumptions, pure point diffraction was characterized via various conditions in \cite{BL,BM,Gouere-1,Gouere-2,ARMA,MS,MoSt,SOL,SOL1}. Recently, a complete characterisation for pure point diffraction, which unifies all these previous results, was given in \cite{LSS,LSS2}.

The great progress done in the study of systems with pure point spectrum, as well as the discovery of many models with interesting long range order and mixed diffraction spectrum (see \cite{BSS,BG2} just to name a few) has shifted the focus on understanding all components of the diffraction spectrum.

The Eberlein decomposition of measures \cite{TAO,Eb,ARMA,MoSt} allows us investigate the pure point and continuous diffraction spectrum in real space by studying the corresponding components of the autocorrelation measure. This was used effectively to show the relatively denseness of the Bragg spectrum for Meyer sets \cite{NS1}, as well as to derive various properties for the (pure point) diffraction spectrum of measures with Meyer set support (see for example \cite{BaGa,BG2,BSS,NS5}). Recently, for Fourier transformable measures with Meyer set support, a complete (or generalized) Eberlein decomposition of the autocorrelation, corresponding via the Fourier transform to the complete spectral decomposition of the diffraction, was established \cite{NS20a,NS21}. This allowed us to establish various properties of the absolutely continuous and singular continuous diffraction measure, respectively, for measures with Meyer set support \cite{NS20a,NS21}.

The existence of the generalized Eberlein decomposition, as well as finding an intrinsic formula for each component, are two important open problems in diffraction theory.
As mentioned above, the generalized Eberlein exists for Fourier transformable measures with Meyer set. In $\R^d$, it was shown in \cite{SS4} that e generalized Eberlein decomposition exists with each component being a tempered distribution of order at most $2d$. It is still unknown if in this case these distributions are given by measures, or if the generalized Eberlein decomposition exists in arbitrary second countable LCAG. The existence of the generalized Eberlein decomposition is equivalent to the following question:

\begin{question}\label{Q1}
Given a Fourier transformable measure $\gamma$, is the measure $(\widehat{\gamma})_{ac}$ a Fourier transform?
\end{question}

The main reason why this question is hard is because it hard in general to decide if a given measure is Fourier transformable or a Fourier transform. Every Fourier transform of a measure is a weakly almost periodic measure and weakly admissible, and every Fourier transformable translation bounded measure is also weakly almost periodic.  Since weakly almost periodic measures have very strong long-range properties (see for example \cite{LS}), it follows that in general, arbitrary measures are not Fourier transformable nor a Fourier transform. A strongly almost periodic measure which is weakly admissible is a Fourier transform \cite{NS20a}, but it is not clear if the same is true for weakly almost periodic measures. Moreover, the issue is muddied by the fact that if $\mu$ is a weakly admissible measure and $\nu$ is an arbitrary measure such that $|\nu| \leqslant |\mu|$, then $\nu$ is also weakly admissible, but $\nu$ is in general not a Fourier transform.

As shown in \cite{SS4}, in $\R^d$ one can overcome some of these issues by working with tempered distributions. One could try to extend this to arbitrary LCAG by using the Schwarz--Bruhat space \cite{Bru}, but unfortunately the Fourier analysis for this space is, to our knowledge, not as advanced as the Fourier theory for tempered distributions or measures.
In this paper, we try an alternate approach, which we explain now.

\smallskip

Given a measure $\omega$, one can study instead the ensemble $\XX(\omega)$ of all measures that locally look like $\omega$. By choosing an ergodic measure $m$ on $\XX(\omega)$ we get an ergodic dynamical system $(\XX(\omega), G, m)$ and hence an unitary representation of $G$ on $\mathcal{H}:= L^2(\XX(\omega), m)$. The dynamical and diffraction spectrum can then be related via the Dworkin argument (see for example \cite{BM,DM,Gouere-2,LMS,LS2}). As observed in \cite{LM,LSS} this leads to the more general concepts of point processes and $\mathcal{N}$-representations, with the former covering a much larger class of examples than just measures. Given an $\mathcal{N}$-representation which satisfies a simple condition, one can define a positive diffraction measure $\sigma$ \cite[Lem.~1.28]{LSS}, which for dynamical systems of translation bounded measures coincides with the usual diffraction measure. Moreover, in this case, one can also define an autocorrelation, which is not necessarily a measure but it is a semi-measure (see Definition~\ref{semi-measure} below). This semi-measure is Fourier transformable (see Definition~\ref{FT-semimeasures}) and its Fourier transform is exactly the diffraction measure $\sigma$ of the $\mathcal N$-representation.

\smallskip

It is the goal of this paper to expand the basic theory of semi-measures introduced in \cite{LSS} and their Fourier theory. Within this setting, a measure $\sigma$ is the Fourier transform of a semi-measure $\vartheta$ if and only if $\sigma$ is a weakly admissible measure (Proposition~\ref{prop:bijection_ftsm_sam}). This implies that each Fourier transformable measure admits a (unique) generalized Eberlein decomposition, with each term being a Fourier transformable semi-measure.

\smallskip

This paper is organised as follows. We start with a brief review of measures, almost periodicity and list the basic properties of weakly admissible measures. In Section~\ref{sect:semi m} we review the definition of semi-measures and give examples of semi-measures. We show in Lemma~\ref{pos implies measure} that any positive semi-measure is a measure. In Section~\ref{sect conv} we introduce the convolution between semi-measures and test functions and discuss the concepts of semi-translation boundedness and intertwining for semi-measures, concepts which play an important role in the rest of the paper. We continue in Section~\ref{FT semi} by reviewing the definition of Fourier transform and Fourier transformability for semi-measures, and study the properties of Fourier transform. We show that every Fourier transformable semi-measure is semi-translation bounded and intertwining. We continue in Section~\ref{sect: pos def} by studying the concept of positive definite semi-measures and proving the Bochner-type theorem for semi-measures (Thm.~\ref{bochner thm}). We prove in Prop.~\ref{prop arma} that a semi-measure, and hence a measure, is Fourier transformable if and only if it is a linear combination of positive definite, intertwining, semi-translation bounded semi-measures. We discuss the almost periodicity of semi-measures and the connection to the Fourier transform in Section~\ref{sect almost per}. In Theorem~\ref{t2}  we prove the existence of the Eberlein decomposition for weakly almost periodic semi-measures, and in Theorem~\ref{genEbe} we show the existence of the generalized Eberlein decomposition for Fourier transformable semi-measures. We discuss in Section~\ref{sect comp} properties of each component of the generalized Eberlein decomposition. We complete the paper by characterising in Theorem~\ref{T1} when a measure $\nu$ is the Fourier transform of a measure, and discussing the generalized Eberlein decomposition for Fourier transformable measures.

\section{Preliminaries}

In this paper, $G$ will always denote a second countable locally compact Abelian group (LCAG). We will denote its Haar measure by $\theta_G$  or simply $\dd x$. We will mainly work with $\Cu(G)$ and $\Cc(G)$ which denote the spaces of uniformly continuous and bounded functions and the space of continuous functions with compact support.

We will denote by $K_2(G)$ and $KL(G)$ the following subspaces of $\Cc(G)$:
\begin{align*}
K_2(G) &:= \mbox{Span} \{ f*g\, :\, f,g \in \Cc(G)\} \,,  \qquad
\reallywidehat{K_2(G)} := \{ \widehat{f}\, :\, f \in K_2(G) \}\\
KL(G)  &:= \{ f \in \Cc(G)\, :\, \widehat{f} \in L^1(G) \} \,, \qquad  \
\reallywidehat{KL(G)}  := \{ \widehat{f}\, :\, f \in KL(G) \}  \,.
\end{align*}
Here, for $f \in L^1(G)$ the \emph{Fourier transform} $\widehat{f}: \widehat{G} \to \C$ and \emph{inverse Fourier transform} $\reallywidecheck{f}: \widehat{G} \to \C$ are defined as usual via
\[
\widehat{f}(\chi)=\int_{G} \overline{\chi(t)}\, f(t)\ \dd t  \qquad \text{ and } \qquad \reallywidecheck{f}(\chi) =\widehat{f}(\overline{\chi}) = \int_{G} \chi(t)\, f(t)\ \dd t  \,.
\]
For general properties of the Fourier transform of $L^1$ functions on $G$ we recommend \cite{rud}.

For any function $g$ on $G$ and $t \in G$, the functions $T_tg, \widetilde{g}$ and $g^{\dagger}$ are defined by
\[
(T_tg)(x):=g(x-t)\,, \qquad \widetilde{g}(x):=\overline{g(-x)} \qquad \text{ and } \qquad g^{\dagger}(x):=g(-x) \,.
\]

A \emph{measure} $\mu$ on $G$ is a linear functional on $C_{\text{c}}(G)$ such that, for every compact subset $K\subseteq G$, there is a constant $a_K>0$ with
\[
|\mu(g)| \leqslant a_{K}\, \|g\|_{\infty}
\]
for all $g\in C_{\text{c}}(G)$ with $\supp(g) \subseteq K$. Here, $\|g\|_{\infty}$ denotes the supremum norm of $g$. By Riesz' representation theorem \cite{Rud}, this definition is equivalent to the classical measure theory concept of regular Radon measure.

Similar to functions, for a measure $\mu$ on $G$ and $t \in G$, we define $T_t\mu, \widetilde{\mu}$ and $\mu^{\dagger}$ by
\[
(T_t\mu)(g):= \mu(T_{-t}g)\,, \qquad \widetilde{\mu}(g):=\overline{ \mu (\widetilde{g})} \qquad  \text{ and } \quad
\mu^{\dagger}(g):= \mu(g^{\dagger}).
\]

Given a measure $\mu$, there exists a positive measure $| \mu|$ such that, for all $f \in \Cc(G)$ with $f \geqslant 0$, we have \cite{Ped} (compare \cite[Appendix]{CRS2})
\[
| \mu| (f)= \sup \{ \left| \mu (g) \right| \ :\ g \in \Cc(G),\,  |g| \leqslant f \} \,.
\]
The measure $| \mu|$ is called the \emph{total variation of} $\mu$.

\begin{definition}
A measure $\mu$ on $G$ is called \emph{Fourier transformable as measure} if there exists a measure $\widehat{\mu}$ on $G$ such that
\[
\reallywidecheck{f}\in L^2(\widehat{\mu}) \qquad \text{ and } \qquad
\left\langle \mu\, , \, f*\widetilde{f} \right\rangle =
\left\langle \widehat{\mu}\, , \, |\reallywidecheck{f}|^2 \right\rangle
\]
for all $f\in\Cc(G)$.

In this case, the measure $\widehat{\mu}$ is called the \emph{Fourier transform} of $\mu$.
We will denote the space of Fourier transformable measures by $\cM_T(G)$.
\end{definition}

\begin{remark}
It was shown in \cite{CRS} that a measure $\mu$ on $G$ is Fourier transformable if and only if there is a measure $\widehat{\mu}$ on $\widehat{G}$ such that
\[
\reallywidecheck{f}\in L^1(\widehat{\mu}) \qquad \text{ and } \qquad
\left\langle \mu\, , \, f \right\rangle =
\left\langle \widehat{\mu}\, , \reallywidecheck{f} \right\rangle
\]
for all $f\in KL(G)$.
\end{remark}

\smallskip

The next property will turn out to be quite useful, when we want to give sufficient conditions for the continuity of the Fourier transform.

\begin{definition}
A measure $\mu$ on $G$ is called \emph{translation bounded} if
\[
\|\mu\|_K:=\sup_{t\in G} |\mu|(t+K)<\infty \,,
\]
for all compact sets $K\subseteq G$.

As usual, we will denote by $\cM^\infty(G)$ the space of translation bounded measures and use the notation
\[
\cM^\infty_{T}(G):=\cM^\infty(G) \cap \cM_{T}(G) \,.
\]
\end{definition}

\medskip

Next, we review the concept of almost periodicity, a natural generalisation of periodicity. For a general review we recommend \cite{Eb,ARMA,LSS,LSS2,MoSt,ST}, just to name a few.

\begin{definition}
A function $f\in\Cu(G)$ is strongly (respectively weakly) almost periodic if the closure of $\{T_tf\, :\, t\in G\}$ in the strong (respectively weak) topology is compact. The spaces of strongly and weakly almost periodic functions on $G$ are denoted by $S\hspace*{-1pt}AP(G)$ and $W\hspace*{-2pt}AP(G)$, respectively.
\end{definition}

For every $f\in W\hspace*{-2pt}AP(G)$ the \emph{mean}
\[
M(f):=\lim_{n\to\infty} \frac{1}{A_n} \int_{A_n+s} f(t)\ \dd t
\]
exists uniformly in $s\in G$, for any given \emph{van Hove sequence} $(A_n)_{n\in\N}$  \cite[Prop. 4.5.9]{MoSt}.

Here, \emph{a van Hove sequence} $\cA=(A_n)_{n\in\N}$ is a sequence of compact Borel subsets of $G$ such that
\[
\lim_{n\to\infty} \frac{|((K+A_n)\setminus A_n^{\circ})\cup((-K+A_n^C)\cap A_n)|}{|A_n|} = 0 \,,
\]
for every compact set $K\subseteq G$.

\medskip

\begin{definition}
A function $f\in W\hspace*{-2pt}AP(G)$ is called null weakly almost periodic if $M(|f|)=0$.
\end{definition}
Let us note here that $|f| \in WAP(G)$ if $f \in WAP(G)$.

These definitions carry over to measures.

\begin{definition}
A measure $\mu\in\cM^{\infty}(G)$ is \emph{strongly}, \emph{weakly} or \emph{null weakly almost periodic} if $f*\mu$ is a strongly, weakly or null weakly almost periodic function, for all $f\in\Cc(G)$. We will denote by
$\SAP(G)$, $\WAP(G)$, and $\WAP_0(G)$ the spaces of strongly, weakly and null
weakly almost periodic measures.
\end{definition}

\section{Weakly admissible measures}

In this section, we briefly review the concept of weakly admissible measures.

\begin{nota}
Recall that (by convention) by $f \in L^1(\mu)$ we mean
\[
\int_G |f(x)|\, \dd|\mu|(x) < \infty \,.
\]
\end{nota}

Let us start with the following definition \cite{NS20b}.

\begin{definition}
A measure $\mu$ on $G$ is called \textit{weakly admissible} if $\reallywidecheck{f}\in L^1(\mu)$ for all $f\in K_2(\widehat{G})$.

We denote the space of weakly admissible measures on $G$ by $\cM^{w}(G)$.
\end{definition}

Let us next recall the following properties for weakly admissible measures, which will be important in the remaining of the paper.

\begin{lemma} \cite[Lem.~3.2]{NS20b} \label{lem:stronglyadmissible_prop}
Let $\mu\in\mathcal{M}(G)$. Then,
\begin{itemize}
  \item[(a)] $\mu$ is weakly admissible if and only if $|\mu|$ is weakly admissible.
  \item[(b)] If $\mu$ is weakly admissible and $| \nu| \leqslant | \mu |$, then $\nu$ is weakly admissible.
  \item[(c)] $\mu$ is weakly admissible if and only if $\mu_{\operatorname{pp}}, \mu_{\operatorname{ac}}, \mu_{\operatorname{sc}}$ are weakly admissible.
  \item[(d)] If $\mu$ is weakly admissible, then $\overline{\mu}$ is weakly admissible.
  \item[(e)] If $\mu$ is weakly admissible, then $\widetilde{\mu}$ and $\mu^\dagger$ are weakly admissible.
  \item[(f)] If $\mu$ is weakly admissible, then $T_t \mu$ is weakly admissible for all $t \in G$.
\end{itemize}  \qed
\end{lemma}

\begin{lemma} \label{lem:wa_sa_prop}
Let $\mu\in\cM(G)$.
\begin{enumerate}
\item [(a)] If $\mu$ is weakly admissible, then $\mu$ is translation bounded.
\item[(b)] If $\mu$ is Fourier transformable, then $\widehat{\mu}$ is weakly admissible.
\end{enumerate}
\end{lemma}
\begin{proof}
(a) This follows from \cite[Thm.~3.3]{NS20b}.

\smallskip

\noindent (b) This follows directly from the definition of Fourier transformable measures.
\end{proof}

\section{Semi-measures}\label{sect:semi m}

As emphasized in the introduction, we plan to use the newly introduced concept of semi-measures \cite{LSS} to show the existence of the generalized Eberlein decomposition for Fourier transformable measures within this class. We do this by showing that within the class of Fourier transformable semi-measures, Question~\ref{Q1} has a positive answer.

\smallskip

Let us first recall the definition of semi-measures from \cite{LSS}.

\begin{definition}\label{semi-measure}
We will call a functional $\vartheta: K_2(G)\to \C$ a \emph{semi-measure} on $G$.

We denote the set of semi-measures on $G$ by $\mathcal{SM}(G)$.
\end{definition}

Let us now look at some simple examples of semi-measures.

\begin{example}\label{ex4}
Let us consider some examples.
\begin{enumerate}
\item[(1)] Obviously, any measure is a semi-measure.
\item[(2)] The mapping
\[
\vartheta:K_2(\R)\to \C \,, \qquad f\mapsto \int_0^{\infty} \reallywidecheck{f}(t) \ \dd t \,,
\]
is a semi-measure but not a measure.
Recall here that (as distributions) one has
\[
\widehat{H}(f)=\frac{1}{2}\delta_0(f) - \frac{1}{2 \pi} \lim_{\delta\downarrow0} \int_{|t|>\delta} \frac{f(t)}{t}\ \dd t \,,
\]
where $H$ is the Heaviside distribution. This means that
\[
\vartheta(f)
    = H (\reallywidecheck{f})= \widehat{H}( f^\dagger)= \frac{1}{2} f(0)+ \frac{1}{2 \pi} \lim_{\delta\downarrow0} \int_{|t|> \delta} \frac{f(t)}{t}\ \dd t
\]
for all $f \in K_2(\R) \cap \Cc^\infty(\R)$.
\end{enumerate}\qed
\end{example}

\medskip
We should emphasize that we do not ask (yet) for any continuity for semi-measures, which makes the class of semi-measures very large. Our approach is to try to
study instead the properties of the Fourier transform of semi-measures and restrict to the subspaces of semi-measures satisfying these particular properties (for example, semi-translation bounded, intertwining, positive definiteness).

\medskip

Let us introduce below an example of a semi-measure which we will use below as an instructive example.

\begin{example}\label{ex2}
Fix some $f \in K_2(G)$, and let $B$ be a (Hamel) basis for $K_2(G)$ as a $\C$-vector space, with the property that $f \in B$.
Then, each element $g \in K_2(G)$ can be written uniquely as
\[
g= Cf+ \sum_{h\in B\setminus\{f\}} c_hh   \,,
\]
with $C,c_h \in \C$ and $C,c_h\neq0$ for at most finitely many $h$.
For each such $g$, define
\[
\varsigma_{f, B}(g)=C \,.
\]
Then, $\varsigma_{f,B}$ is a semi-measure.
\end{example}

\begin{example}
Let $G$ be a LCAG, let $F(G)$ be the vector space of all functions $f: G \to \C$, and let $\theta: F(G) \to \C$ be any linear functional.
Since $K_2(G)$ is a subspace of $F(G)$, the restriction of $\theta$ to $K_2(G)$ is a semi-measure.
\end{example}

\begin{example}\label{ex distributions}
Let $\Psi \in \cD'(\R^d)$ be any distribution. Then,
\[
K_2(\R^d) \cap \Cc^\infty(\R^d) \ni f \mapsto \Psi(f)
\]
defines a linear function on the subspace $K_2(\R^d)\, \cap\, \Cc^\infty(\R^d)$ of $K_2(\R^d)$, which can be extended (non-uniquely) to a semi-measure $\vartheta$.
This shows that, for each distribution $\Psi \in \cD'(\R^d)$, there exists a semi-measure $\vartheta$ such that
\[
\vartheta(f)= \Psi(f) \qquad \text{ for all } f \in K_2(\R^d) \cap \Cc^\infty(\R^d) \,.
\]
\end{example}

\begin{example}
Let $\Phi : L^1(\widehat{G}) \to \C$ be any linear functional. Then $\vartheta : K_2(G) \to \C$ given by
\[
\vartheta(g)=\Phi(\widehat{g})
\]
is a semi-measure.
\end{example}

As with functions and measures, we can define the operators $\widetilde{}\,, \dagger$ and $T_t$ for all $t \in G$ on the space $\SM(G)$ via
\[
(T_t\vartheta)(g):= \vartheta(T_{-t}g)\,, \qquad \widetilde{\vartheta}(g):=\overline{ \vartheta (\widetilde{g})} \qquad  \text{ and } \quad
\vartheta^{\dagger}(g):= \vartheta(g^{\dagger}).
\]

\medskip

Next, we want to introduce a topology on $\cSM(G)$. As so far we did not ask for any type of continuity for semi-measures, the most natural topology to consider is the topology of pointwise convergence.

To this regard, for each $f \in K_2(G)$ define $p_f : \SM(G) \to [0, \infty)$ via
\[
p_f(\vartheta)=\left| \vartheta(f) \right| \,.
\]
It is easy to see that each $p_f$ defines a semi-norm on $\SM(G)$, and that the family $\{ p_ f : f \in K_2(G)\}$ separates points.

\begin{definition}
The locally convex topology on $\SM(G)$ defined by the family of semi-norms $\{ p_f :f \in K_2(G)\}$ is denoted by $\tau_{\text{pt}}$ and called the \emph{pointwise topology for semi-measures}.
\end{definition}

Note that in this topology a net $\vartheta_\alpha \in \SM(G)$ converges to some $\vartheta \in \SM(G)$ if and only if $\lim_\alpha \vartheta_\alpha(f)=\vartheta(f)$ for all $f \in \SM(G)$.

\medskip

The following result is standard and immediate, and we skip its proof.
\begin{proposition}\label{sm complete}
The space $\cSM(G)$ is a complete locally convex topological vector space with respect to pointwise topology. \qed
\end{proposition}


\medskip

As pointed before, $K_2(G) \subseteq \Cc(G)$ implies that any measure is a semi-measure. It is natural to ask when a semi-measure can be extended to a measure, in which case we will simply say that the semi-measure is a measure. Since measures need to be continuous in the inductive topology, and the definition of semi-measures requires no continuity whatsoever, it is clear that not all semi-measures can be extended to a measure.
We show below that a semi-measure is a measure exactly when it is continuous in the topology induced from the inductive topology via the embedding $K_2(G) \subseteq \Cc(G)$.
In particular, it follows that any positive semi-measure is a measure, and that a semi-measure is a measure exactly when it is a linear combination of positive semi-measures.

\begin{lemma} \cite[Lem. C.8]{LSS} \label{semi is measure}
Let $\vartheta:K_2(G)\to \C$ be a semi-measure. Then, $\vartheta$ is a measure if and only if, for all compact sets $K \subseteq G$, there exists a constant $C_K>0$ such that
\[
| \vartheta(f) | \leqslant C_K \| f\|_\infty
\]
for all $f \in K_2(G)$ with $\supp(f) \subseteq K$.   \qed
\end{lemma}

\begin{remark}
Consider the mapping
\[
D : \Cc^\infty(\R)\cap K_2(\R) \to \C \,, \qquad \varphi\mapsto\varphi'(0) \,,
\]
and let $\vartheta : K_2(\R) \to \C$ be any extension of this, compare Example \ref{ex distributions}.
Then, $\vartheta$ cannot be extended to a measure. Indeed,
\[
\vartheta(f)= D(f)\qquad \text{ for all } f \in K_2(\R) \cap \Cc^\infty(\R) \,.
\]

Now, pick some $f \in K_2(\R) \cap \Cc^\infty(\R)$ such that $f'(0)=1$ and $\supp(f) \subseteq [-1,1]$. Let $f_n(x):=f(nx)$. Then, $\supp(f_n) \subseteq [-1,1]$ and $\|f_n \|_\infty = \|f\|_\infty \neq 0$.
\[
| \vartheta(f_n) |=| D(f_n)|= | -f_n'(0) |= | -nf'(0) |=n= \frac{n}{\|f\|_\infty} \| f_n \|_\infty  \,,
\]
and hence, by Lemma~\ref{semi is measure}, $\vartheta$ cannot be extended to a measure.
\end{remark}

Next, we show that a semi-measure is a translation bounded measure if and only if the relation in Lemma~\ref{semi is measure} is uniform in translates.

\begin{lemma}\label{lem 2}
Let $\vartheta:K_2(G)\to \C$ be a semi-measure. Then, $\vartheta$ is a translation bounded measure if and only if, for all compact sets $K \subseteq G$, there exists a constant $C_K>0$ such that
\[
| \vartheta(T_tf) | \leqslant C_K\, \| f\|_\infty \,.
\]
for all $f \in K_2(G)$ with $\supp(f) \subseteq K$ and all $t \in G$.
\end{lemma}
\begin{proof}
$\Longrightarrow$: This follows from Lemma~\ref{semi is measure}.

\medskip

\noindent $\Longleftarrow$: By Lemma~\ref{semi is measure}, there exists some measure $\mu$ such that
\[
\vartheta(f)=\mu(f)
\]
for all $f \in K_2(G)$.
Next, fix some open pre-compact set $U$, and let $K$ be the closure of $-U$. Set
\[
\cF:= \{ f \in K_2(G)\, :\, \supp(f) \subseteq -U,\, \|f\|_\infty \leqslant 1 \} \,.
\]
Then, one has
\[
| (\mu*f)(t)|= | \mu (T_tf^\dagger)| \leqslant C_K\, \| f \|_\infty \leqslant C_K
\]
for all $f \in \cF$ and all $t \in G$, and hence
\[
\|f*\mu \|_\infty \leqslant C_K \,.
\]
Then, \cite[Cor.~3.4]{SS2} gives
\[
\| \mu \|_U \leqslant C_K < \infty \,,
\]
which completes the claim.
\end{proof}

Later, after we talk about convolution for measures, we will give other characterizations for when a distribution is a translation bounded measure, similar to \cite{SS2}.

\smallskip

Next, exactly as for measures, we can prove that positivity for semi-measures implies continuity in the inductive topology. In particular, any positive semi-measure is a measure.

\begin{lemma}\label{pos implies measure}
Let $\vartheta:K_2(G)\to \C$ be a positive semi-measure. Then, $\vartheta$ can be extended to a measure.
\end{lemma}
\begin{proof}
The proof is similar to the proof of Theorem C1 in \cite{CRS2}.

Let $K \subseteq G$ be a compact set. Then, there exists some $g \in K_2(G) $ such that $g \equiv 1$ on $K$, see \cite[Lem.~3.1]{NS20a}. Also, for every $f\in K_2(G)$ with $\operatorname{supp}(f)\subseteq K$, we have
\begin{equation}\label{eq12}
|\operatorname{Re}f|,\, |\operatorname{Im}f| \leqslant \|f\|_{\infty}\, g \,.
\end{equation}

Next, we show that $\operatorname{Re}f, \operatorname{Im}f \in K_2(G)$. This is needed to make sure that $\vartheta(\operatorname{Re}f), \vartheta(\operatorname{Im}f)$ make sense. Note that by definition, any $f \in K_2(G)$ can be written as
\[
f=\sum_{k=1}^n c_k\, (g_k*h_k) \,,
\]
for some $c_k \in \C, g_k, h_k \in \Cc(G)$. Then, $\overline{f}=\sum_{k=1}^n \overline{c_k}\,  (\overline{g_k}* \overline{h_k}) \in K_2(G)$, and we obtain
\[
 \operatorname{Re}f   =\frac{1}{2} \left(f+\overline{f} \right) \in K_2(G) \qquad \text{ and } \qquad
 \operatorname{Im}f   =\frac{1}{2\im} \left(f-\overline{f} \right) \in K_2(G) \,.
\]
Now, since $\operatorname{Re}f, \operatorname{Im}f \in K_2(G)$ and $\vartheta$ is positive, \eqref{eq12} gives
\[
-\vartheta(\|f\|_{\infty}g) \leqslant \vartheta(\operatorname{Re}f),\, \vartheta(\operatorname{Im}f) \leqslant \vartheta(\|f\|_{\infty}g) \,.
\]
This immediately implies $|\vartheta(\operatorname{Re}f)|,\, |\vartheta(\operatorname{Im}f)| \leqslant \vartheta(\|f\|_{\infty} g)$. Therefore, we obtain
\[
|\vartheta(f)| \leqslant |\vartheta(\operatorname{Re}f)| + |\vartheta(\operatorname{Im}f)| \leqslant 2\, \vartheta (\|f\|_{\infty}g) = 2 \vartheta (g) \, \|f\|_{\infty} \,.
\]
Since $c_K:= 2\, \vartheta(g)$ only depends on $K$ (and not on $f$), $\vartheta$ is continuous with respect to the inductive topology.

Finally, $\vartheta$ can be uniquely extended to a continuous functional $\vartheta:\Cc(G)\to \C$, because $K_2(G)$ is dense in $\Cc(G)$. Therefore, $\vartheta$ is a measure.
\end{proof}

The next result is an immediate consequence.

\begin{corollary}\label{cor 1}
A semi-measure  $\vartheta:K_2(G)\to \C$ is a measure if and only if it is a finite linear combination of positive semi-measures.  \qed
\end{corollary}

\section{Convolutions}\label{sect conv}

In this section, we introduce the convolution of a semi-measure $\vartheta$ and a function $\phi \in K_2(G)$, and study some concepts around this convolution, which will play an important role latter.

\smallskip

Given a semi-measure $\vartheta$ and some $f \in K_2(G)$, we can define a function $\vartheta*f :G \to \C$, called \emph{convolution of $\vartheta$ and $f$} via
\begin{equation}\label{semi conv}
(\vartheta*f)(t)=\vartheta(T_t f^\dagger) \,.
\end{equation}
Note here that if $\vartheta$ is a measure, then \eqref{semi conv} coincide with the definition of the convolution between the measure $\vartheta$ and the function $f \in K_2(G) \subseteq \Cc(G)$. Moreover, if $\vartheta$ is a semi-measure on $\R^d$ which coincides on $K_2(\R^d) \cap \Cc^\infty(\R^d)$ with some distribution $\omega$, then we have $\vartheta*f=\omega*f$ for all $f \in K_2(\R^d) \cap \Cc^\infty(\R^d)$.

We will omit the proof of the next lemma, since all computations are straightforward.

\begin{lemma}  \label{lem:popert_sm}
\begin{itemize}
\item [(a)] $*: \SM(G) \times K_2(G) \to F(G)$ is a bilinear form.
\item [(b)] For all $t \in G, \vartheta \in \SM(G)$ and $f \in K_2(G)$, we have
$(T_t \vartheta)*f=\vartheta*(T_tf)=T_t(\vartheta*f)$.

\item [(c)] For all $\vartheta \in \SM(G)$ and $f \in K_2(G)$, we have
$(\vartheta*f)^\dagger=\vartheta^\dagger*f^\dagger$.
\item [(d)] For all $\vartheta \in \SM(G)$ and $f \in K_2(G)$, we have
$\widetilde{\vartheta*f}=\widetilde{\vartheta}*\widetilde{f}$.   \qed
\end{itemize}
\end{lemma}

In the case that $\vartheta$ is a measure, the convolution $\vartheta*f$ is a continuous function for all $f \in K_2(G)$. Similarly, if $\vartheta$ is a distribution and $f \in K_2(\R^d) \cap \Cc^\infty(\R^d)$, the convolution $\vartheta*f$ is infinitely differentiable, and hence continuous. The same is not true for semi-measures, as Example~\ref{ex1}
below shows. To make the details simpler, let us start with the following simple lemma.

\begin{lemma}\label{lin ind}
Let $f \in \Cc(\R)$ be any non-zero function. Then, $\{ T_tf :t \in \R \}$ is linearly independent.
\end{lemma}
\begin{proof}
Assume by contradiction that the set is linearly dependent. Then, there exists some $t_1 <t_2<\ldots< t_n \in \R$ and $c_1,\ldots, c_k \in \C$ such that
$\sum_{j=1}^n c_j T_{t_j}f =0$.
This means that
\begin{equation}\label{eq2}
\sum_{j=1}^n c_j f(x-t_j) =0
\end{equation}
for all $x \in \R$.

Let $a := \inf \{ x \in \R : f(x) \neq 0 \}$, which is a real number since $f \in \Cc(\R^d)$ is not the zero function. Then, there exists some $b \in \R$ such that
\[
a < b <a+t_2-t_1 \quad \mbox{ and }\quad f(b) \neq 0 \,.
\]
Setting $x=b+t_1$ in \eqref{eq2} gives
\[
c_1f(b)+c_2f(b+t_1-t_2)+c_3f(b+t_1-t_3)+...+c_nf(b+t_1-t_n)=0 \,.
\]
Now, $b< a+t_2-t_1$ implies $b+t_1-t_2<a$. Hence, we obtain
\[
b+t_1-t_n \leqslant b+t_1-t_2 <a
\]
for all $2 \leqslant j \leqslant n$.
Therefore, the definition of $a$ gives
\[
f(b+t_1-t_2)=f(b+t_1-t_3)=\ldots=(b+t_1-t_n)=0 \,,
\]
and hence $c_1f(b)=0$.
But this contradicts the fact that $c_1 \neq 0$ and $f(b) \neq 0$.
\end{proof}

\begin{example}\label{ex1}
Fix some $0 \neq f \in K_2(G)$, and let $B$ be a (Hamel) basis for $K_2(G)$ as a $\C$-vector space, with the property that $T_tf \in B$ for all $t \in G$. Such a basis exists by Lemma~\ref{lin ind}. Let $\varsigma_{f,B}$ be the semi-measure from Example~\ref{ex2}.
Then,
\[
(\varsigma_{f,B}*f^\dagger)(t)=
\begin{cases}
1 \,, & \mbox{ if } t=0 \,, \\
0 \,, & \mbox{ otherwise}\,.
\end{cases}
\]
In particular, $\varsigma_{f,B}*f^\dagger$ is not continuous.
\end{example}
\begin{proof}
For simplicity let $B'$ be such that
\[
B= \{ T_tf\, :\, t \in \R \} \sqcup B' \,,
\]
that is
\[
B' =B \backslash \{ T_tf\, :\, t \in \R \} \,.
\]
Now, let $t \in G$ be fixed but arbitrary. Then,
\[
(\varsigma_{f,B}*f^\dagger)(t)=  \varsigma_{f,B}(T_tf) \,.
\]
If $t=0$, this gives
\[
(\varsigma_{f,B}*f^\dagger)(t)=  \varsigma_{f,B}(f)=1
\]
by definition. Next, if $t \neq 0$, we have
\[
T_tf= 1 \cdot T_tf+ \sum_{s \in \R, s\neq t} 0\cdot T_sf+ \sum_{ g \in B'} 0 \cdot g
\]
and hence, by definition
\[
(\varsigma_{f,B}*f^\dagger)(t)=  \varsigma_{f,B}(T_tf)=0 \,.
\]
\end{proof}

\begin{remark} The semi-measure $\varsigma_{f,B}$ from Example~\ref{ex1} is not a measure.
\end{remark}

Next, let us introduce the following definition.

\begin{definition}
A semi-measure $\vartheta$ is called \emph{semi-translation bounded} if
\[
\vartheta*f \in \Cu(G) \qquad \text{ for all } f \in K_2(G) \,.
\]
We denote the space of semi-translation bounded semi measures by $\SM^\infty(G)$.
\end{definition}

The definition of translation boundedness for measures usually only requires that the convolution $\mu*f$ is bounded for all $f \in \Cc(G)$.
This is because in this case, as mentioned above, $\mu*f$ is automatically a continuous function. Moreover, \cite[Thm.~1.1]{ARMA1} implies that, if $\mu*f$ is bounded for all $f \in \Cc(G)$, it is also uniformly continuous. Therefore, requiring for a measure that the convolution $\mu*f$ is bounded for all $f \in \Cc(G)$ is equivalent to the condition $\mu*f \in \Cu(G)$ for all $f \in \Cc(G)$.

Similarly, a tempered distribution $\omega$ is called translation bounded \cite{ST} if the convolution $\omega*f$ is bounded for all $f \in \cS(\R^d)$. Again, in this case we do not need to worry about uniform continuity since we get it for free. Indeed, by \cite[Prop.~2.1]{ST}, a tempered distribution is translation bounded if and only if $\omega*f \in \Cu(\R^d)$ for all $f \in \cS(\R^d)$.

Now, due to the lack of continuity in the definition of semi-measures, the same does not seem to be true for this class, and we want to add the continuity as part of the definition. This leads to the following natural question:

\begin{question}
Does there exists a semi-measure $\vartheta$ which is not bounded, but such that for all $g \in K_2(G)$ the function $\vartheta*g$ is bounded?
\end{question}

We suspect that the answer is yes, and that the semi-measure $\varsigma_{f,B}$ in Example~\ref{ex1} provides such an example for the right choice of a Basis $B$, but due to the mystic nature of Hamel bases for this space we could not construct an explicit example. As $\varsigma_{f,B}*f^\dagger$ is bounded but not continuous, if one can show that there exists a Hamel basis $B$ such that $\varsigma_{f,B}*g$ is bounded for all $g \in K_2(G)$, the claim will follow. Note that if $K_2(G)$ admits a basis $B$ which is translation invariant, meaning that for all $t \in \R$, $g \in B$ we have $T_tg \in B$, then such an example is trivial to construct.

The following result is an immediate consequence of the definition and Lemma~\ref{lem:popert_sm}.

\begin{lemma}
Let $\vartheta \in \SM(G)$. Then, $\vartheta \in \SM^\infty(G)$ if and only if there exists some linear operator $T : K_2(G) \to \Cu(G)$ which commutes with translates, such that
\[
\vartheta(f)= (Tf)(0)\qquad \text{ for all } f \in K_2(G) \,.
\]
Moreover, in this case we have
\[
T(f)= \vartheta*f^\dagger \qquad  \text{ for all } f \in K_2(G) \,.
\]
In particular, $T$ is unique.  \qed
\end{lemma}

Note here that the operator $T$ may not be bounded.
We will denote this operator by $T_\vartheta$ and refer to it as the \emph{operator induced by $\vartheta$}.

The next result is straight forward.

\begin{lemma}\label{lemms smtb} Let $\mu$ be a measure.
\begin{itemize}
\item[(a)] If $\mu$ is translation bounded, then it is semi-translation bounded as a semi-measure.
\item[(b)] If $\mu$ is positive definite, then it is semi-translation bounded as a semi-measure.
\item[(c)] If $\mu$ is positive, then $\mu$ is semi-translation bounded as a semi-measure if and only if $\mu \in \cM^\infty(G)$.
\end{itemize}
\end{lemma}
\begin{proof}
(a) This follows from $K_2(G) \subseteq \Cc(G)$.

\medskip

\noindent(b) This follows from \cite[Lem.~4.9.24]{MoSt} and the polarisation identity (see for example \cite[p. 244]{MoSt}).

\medskip

\noindent (c) $\Longleftarrow$: This follows from (a).

\smallskip

\noindent $\Longrightarrow$: Let $K \subseteq G$ and let $f \in K_2(G)$ be such that $f \geqslant 1_K$. Then, we have
\[
| \mu|(t+K)= \mu(t+K) \leqslant \mu(T_{t}f) = (\mu*f^\dagger)(t) \leqslant \| \mu*f^\dagger \|_\infty \,.
\]
for all $t \in G$. This immediately gives
\[
\| \mu \|_K \leqslant  \| \mu*f^\dagger \|_\infty  < \infty \,.
\]
\end{proof}

\begin{remark}

The converse of Lemma~\ref{lemms smtb} (a) is not true. Indeed, by \cite{ARMA1}, there exists some measure $\mu$ which is positive definite but not translation bounded \cite[Prop. 7.1]{ARMA1}. Such a measure must be semi-translation bounded by Lemma~\ref{lemms smtb}, but it is not translation bounded as a measure.
\end{remark}

\medskip

We can now state and prove the following extended characterization of translation bounded measures semi-measures.
First, fix an open pre-compact set $U \subseteq G$, and set
\begin{align*}
F_B(G)&:= \{ f :G \to \C : \|f \|_\infty < \infty \} \,, \\
K_2(U) &:= \{ f \in K_2(G): \supp(f) \subseteq U \} \,.
\end{align*}
Note here that $(F_B(G),  \| \cdot \|_\infty)$ is a Banach space, while $(K_2(U), \| \cdot \|_\infty)$ is a normed space, which is typically not complete.

Let $K_2^U(G)$ denote the unit ball in $K_2(U)$, that is
\begin{equation*}
K_2^{U}(G):=\{ f \in K_2(G)\, :\, \supp(f) \subseteq U,\, \|f \|_\infty \leqslant 1 \} \,.
\end{equation*}

\begin{proposition}\label{prop tb}
Let $\vartheta \in \cSM(G)$, and let $U$ be a fixed open pre-compact set. Then, the following statements are equivalent.
\begin{itemize}
\item[(i)] $\vartheta$ can be extended to a translation bounded measure.
\item[(ii)] For all compact sets $K \subseteq G$, there exists a constant $C_K>0$ such that
\[
| \vartheta(T_tf) | \leqslant C_K\, \| f\|_\infty
\]
for all $f \in K_2(G)$ with $\supp(f) \subseteq K$ and all $t \in G$.
\item[(iii)] One has
\begin{equation}\label{eq11}
\sup \{ \| \vartheta*f \|_\infty\, :\, f \in  K_2^{U}(G) \} <\infty \,.
\end{equation}
\item[(iv)] The operator
\[
  T_\vartheta: ( K_2(U),  \| \cdot \|_\infty) \to (F_B(G), \| \cdot \|_\infty) \,,\qquad  T_\vartheta(f)= \vartheta*f
\]
is bounded.
\item[(v)] $\vartheta$ is semi-translation bounded, and the operator
\[
T_\vartheta: ( K_2(U),  \| \cdot \|_\infty) \to (\Cu(G), \| \cdot \|_\infty) \,, \qquad T_\vartheta(f)= \vartheta*f
\]
is bounded.
\item[(vi)] $\vartheta$ is a finite linear combination of positive semi-measures in $\SM^\infty(G)$.
\end{itemize}
\end{proposition}
\begin{proof}
(i)$\iff$(ii) This is a consequence of Lemma~\ref{lem 2}.

\medskip

\noindent (ii)$\implies$(iii) Let $K$ be any compact set containing $-U$. Then, $\supp(f^\dagger) \subseteq K$ for all $f \in  K_2^{U}(G)$, and  by (ii), we have
\[
| (\vartheta*f)(t) | = | \vartheta(T_tf^\dagger)| \leqslant C_K \|f ^\dagger \|_\infty = C_K
\]
for all $t \in G$. This shows that
\[
\| \vartheta*f \|_\infty \leqslant C_K \qquad \text{ for all } f \in K_2(U) \,.
\]

\smallskip

\noindent (iii)$\implies$(i) Fix a compact set $K \subseteq U$ with non-empty interior. It is easy to see that
\[
\cF_t:= \{ f \in K_2(G)\, :\, \supp(f) \subseteq t-K \}
\]
is dense in the Banach space
\[
C(G:t-K)=\{ g \in \Cc(G)\, :\, \supp(g) \subseteq t-K \}
\]
for all $t \in G$. Eq.~\eqref{eq11} gives that  $f\mapsto \vartheta*f$ is a bounded operator on $\cF_t$. Therefore, $f \mapsto \vartheta(f)$ is also bounded, and hence, it extends uniquely to a continuous operator
\[
\mu: C(G:t-K) \to \C \,.
\]
Since $K$ has non-empty interior, by a standard partition of unity argument, $\mu$ extends to a measure $\mu \in \cM(G)$. Moreover, since this is an extension of $\vartheta$, we have
\[
\sup \{ \| \mu*f \|_\infty\, :\, f \in  K_2^{U}(G) \} <\infty  \,.
\]
Finally, by \cite[Cor.~3.4]{SS2}, we have
\[
\| \mu \|_{-U} =\sup \{ \| \mu*f \|_\infty\, :\, f \in  K_2^{U}(G) \} <\infty  \,.
\]
This proves that $\mu \in \cM^\infty(G)$.

\medskip

\noindent (iii)$\implies$(iv) This follows immediately from the fact that $K_2^U(G)$ is the unit ball in $K_2(G)$.

\medskip

\noindent (i)$\implies$(v)$\implies$(iv) This is obvious.

\medskip

\noindent (i)$\implies$(vi) We know that $\vartheta$ can be extended to a translation bounded measure $\mu$. Since $\mu$ is a translation bounded measures, there exist positive translation bounded measures $\mu_1,\mu_2,\mu_3,\mu_4$ such that
\[
\mu=\mu_1-\mu_2+\im(\mu_3-\mu_4) \,.
\]
Denoting by $\vartheta_j$ the restriction of $\mu_j$ to $K_2(G)$, the claim follows.

\medskip

\noindent (vi)$\implies$(i) This follows from Lemma~\ref{pos implies measure} and  Lemma~\ref{lemms smtb}(c).
\end{proof}

\subsection{Intertwining semi-measures}

By \cite[Lem.~1.28]{LSS}, every $\mathcal N$-representation admits a diffraction measure, or equivalently an autocorrelation (which is a semi-measure) if and only if it is intertwining. In this subsection we extend this concept to arbitrary semi-measures. We will see in the next sections that this concept is crucial for the Fourier theory of semi-measures.

We start by showing that the following conditions are equivalent, either of them can be used as the definition of intertwining.

\begin{lemma}\label{intert}
Let $\vartheta \in \SM^\infty(G)$. Then, the following statements are equivalent.
\begin{itemize}
\item[(i)] For all $f,g \in \Cc(G)$, we have
\[
\left( \vartheta*f \right)*g=\left( \vartheta*g \right)*f
\]
\item[(ii)] For all $f,g \in \Cc(G)$, we have
\[
\left( \vartheta*f \right)*g= \vartheta * (f*g) \,.
\]
\end{itemize}
\end{lemma}
\begin{proof}
(ii)$\implies$(i) This is immediate. Indeed, (ii) gives
\[
\left( \vartheta*f \right)*g = \vartheta * (f*g)  \qquad \text{ and } \qquad
\left( \vartheta*g \right)*f= \vartheta * (g*f)  \,.
\]
As $f*g=g*f$, the claim follows.

\medskip

\noindent (i)$\implies$(ii) Fix $f,g,h \in K_2(G)$. Then, (i) implies
\[
\left( \vartheta*h \right)*(f*g)=\left( \vartheta*(f*g) \right)*h
\qquad \text{ and } \qquad
\left( \vartheta*h \right)*f= \left( \vartheta*f \right)*h \,.
\]
Next, since $\vartheta \in \SM^\infty(G)$, we have  $\vartheta*h,  \vartheta*f \in \Cu(G)$. Thus, as $f,g,h \in K_2(G) \subseteq \Cc(G)$ the following convolutions of continuous functions make sense and
\begin{align*}
( \vartheta*(f*g) )*h
    &=( \vartheta*h )*(f*g)
     =(( \vartheta*h )*f)*g
     =( ( \vartheta*f )*h)*g  \\
    &= ( \vartheta*f )*(h*g)
     =( \vartheta*f )*(g*h)
     = ( ( \vartheta*f )*g)*h \,.
\end{align*}
Therefore, $ \vartheta*(f*g) ,   \left( \vartheta*f \right)*g \in \Cu(G)$, and for all $h \in K_2(G)$, we have
\[
( \vartheta*(f*g) )*h= ( ( \vartheta*f )*g)*h \,.
\]
Replacing $h$ by an approximate identity gives (ii).
\end{proof}

We can now introduce the following definition.

\begin{definition}
A semi-measure $\vartheta \in \SM^\infty(G)$ is called \emph{intertwining} if the equivalent conditions of Lemma~\ref{intert} hold.
\end{definition}

One could define intertwining for semi-measures which are not semi-translation bounded. In this case, one needs to take into account that $\vartheta*f$ could be non-measurable. For this reason, we include semi-translation boundedness as part of the definition. We will give in Example~\ref{ex7} an example of a semi-measure $\vartheta$ and two function $f,g \in K_2(G)$ such that $\vartheta*(f*g)$ and $(\vartheta*f)*g$ make sense but they are not equal.

\smallskip
\medskip

A natural question to ask at this point is the following:

\begin{question}\label{q4} Is every semi-translation bounded semi-measure $\vartheta$ intertwining?
\end{question}

We suspect that the answer is negative, but we could not construct an example. We will see below that every measure is intertwining as a semi-measure.

\smallskip

\begin{remark}
Let $\vartheta \in \SM^\infty(G)$, and let $T_\vartheta$ be the induced operator. Then, $\vartheta$ is intertwining if and only if $T_\vartheta$ commutes with convolution, that is
\[
\left(T_\vartheta f\right)*g= T_{\vartheta}(f*g)
\]
for all $f ,g \in K_2(G)$.
\end{remark}

\begin{example}\label{ex7}
Let $f,g \in K_2(\R)$ be such that $f,g \geqslant 0$, $g \in \Cc^\infty(\R)$ with $\int_{\R} g(t)\,\dd t=1$ and $f$ is differentiable at all $x \neq 0$ but not at $x=0$. Then, $f*g \in \Cc^\infty(\R)$.

We claim that $A:=\{T_t f^\dagger : t \in \R\} \cup \{ T_t (f*g)^\dagger :t \in \R \}$ is linearly independent over $\C$. Indeed, assume by contradiction that $A$ is not linearly independent. Then, there exists a non-trivial linear combination of elements in $A$ which is zero. By Lemma~\ref{lin ind}, this linear combination cannot consist of only elements in $\{T_t f^\dagger : t \in \R\}$, nor only of elements in $\{ T_t (f*g)^\dagger :t \in \R \}$. Therefore, there exist some $m,n \geqslant 1$, $t_1<t_2< \ldots < t_n$ and $s_1< \ldots < s_m$ and non-zero $C_1,C_2, \ldots, C_n, D_1, D_2, \ldots , D_m \in \C$ such that
\[
\sum_{k=1}^n C_k T_{t_k}(f^\dagger) +\sum_{j=1}^m D_j T_{s_j}(f*g)^\dagger =0 \,.
\]
This gives
\[
\sum_{k=2}^n C_k T_{t_k}(f^\dagger) +\sum_{j=1}^m D_j T_{s_j}(f*g)^\dagger =-C_1 T_{t_1}(f^\dagger)  \,.
\]
But this is not possible, as the left hand side is a function which is differentiable at $t_1$, while the right hand side is not differentiable at $t_1$.

Let $V := \mbox{Span}(A)$. We can define a linear functional $\theta:V \to \C
$ by
\[
\theta(T_tf^\dagger)  =1 \qquad  \text{ and } \qquad
\theta(T_t (f*g)^{\dagger})=0
\]
for all $t$.
Then, $\theta$ can be extended (non-uniquely) to a semi-measure $\vartheta : K_2(G) \to \C$. By construction,  we have
\[
(\vartheta*f)(t) =1  \,, \qquad ((\vartheta*f)*g)(t) =1 \qquad \text{ and } \qquad
(\vartheta*(f*g))(t) =0
\]
for all $t \in \R$.

If one can show that $\theta$ can be extended to a semi-translation bounded semi-measure $\vartheta$, then this would provide an example of a semi-measure which is semi-translation bounded and not intertwining, thus answering Question~\ref{q4}.
\end{example}

\begin{remark}
With $\theta$ and $V$ as in Example~\ref{ex7}, a standard application of Zorn's lemma shows that there exists a maximal pair $(W, \theta_W)$ consisting of a subspace $V \subseteq W \subseteq K_2(\R)$ and  a linear mapping $\theta_W: W \to \C$ with the following properties
\begin{itemize}
\item{} $W$ is translation invariant.
\item{} For all $h \in V$, one has $\theta_W(h)=\theta(h)$.
\item{} For all $h \in W$ and $t \in \R$, one has $\theta_{W}(T_th)=\theta_{W}(h)$. 
\end{itemize}
If one can show that the maximality implies that $W=K_2(\R)$, then this would give an example of a semi-translation invariant semi-measure which by Example~\ref{ex7} would not be intertwining.
\end{remark}

The following result is a well known immediate consequence of the Fubini theorem.

\begin{lemma}
Let $\mu$ be a translation bounded measure. Then, $\mu$ is semi-translation bounded and intertwining as a semi-measure. \qed
\end{lemma}

\section{Fourier transformable semi-measures}\label{FT semi}

Now that we have covered some basic properties of semi-measures, we can focus on the subset which will be of utmost importance for the rest of this paper, namely Fourier transformable semi-measures.
Let us first review the following definition of \cite{LSS}.

\begin{definition}\label{FT-semimeasures}
A semi-measure $\vartheta$ is called \emph{Fourier transformable} if there is a measure $\nu$ on $\widehat{G}$ such that
\[
|\reallywidecheck{f\, }|^2 \in L^1(\nu) \qquad \text{ and } \qquad \vartheta(f*\widetilde{f})=\nu(|\reallywidecheck{f\, }|^2 )
\]
for all $f\in \Cc(G)$. We call $\nu$ the \emph{Fourier transform} of $\vartheta$ and write $\nu=\widehat{\vartheta}$. We denote the class of Fourier transformable semi-measures by $\SM_{T}(G)$.
\end{definition}

Let us recall the following results of \cite{LSS}

\begin{proposition} \cite[Prop. C.4]{LSS} \label{prop:bijection_ftsm_sam}
The Fourier transform is a bijection between Fourier transformable semi-measures and weakly admissible measures on $\widehat{G}$. \qed
\end{proposition}

\begin{lemma} \cite[Lem. C.1]{LSS}  \label{lem:properties_sm_1}
Let $\vartheta$ be a Fourier transformable semi-measure.
\begin{enumerate}
\item[(a)] For all $f\in K_2(G)$, we have
\[
\reallywidecheck{f}\in L^1(|\widehat{\vartheta}|) \qquad \text{ and } \qquad \vartheta(f) = \widehat{\vartheta}(\reallywidecheck{f}) \,.
\]
\item[(b)] For all $f\in K_2(G)$, we have
\[
(\vartheta*f)(t) = \int_{\widehat{G}} \chi(t)\, \widehat{f}(\chi)\, \dd\widehat{\vartheta}(\chi)   = \reallywidecheck{\widehat{f}\widehat{\vartheta}}(t) \,.
\]\qed
\end{enumerate}
\end{lemma}

As an immediate consequence of Lemma~\ref{lem:properties_sm_1}, we get the following result, which is the reason why in the previous section we focused on these concepts.

\begin{corollary}\label{cor FT implies tb and inter}
Let $\vartheta$ be a Fourier transformable semi-measure. Then, $\vartheta \in \SM^\infty(G)$ and $\vartheta$ is intertwining.
\end{corollary}
\begin{proof}
First note that $\vartheta*f = \reallywidecheck{\widehat{f}\widehat{\vartheta}}$ for all $f \in K_2(G)$ by Lemma~\ref{lem:properties_sm_1}. Now, let $f \in K_2(G)$ and let $\mu :=\widehat{\vartheta}$. Then, Lemma~\ref{lem:properties_sm_1} implies  $\reallywidecheck{f} \in L^1(|\mu|)$ and hence $\reallywidecheck{f} \mu$ is a finite measure. Thus, by \cite[Lem. 4.8.3]{MoSt},
$\vartheta*f = \reallywidecheck{\widehat{f}\mu} \in \Cu(G)$. This proves that $\vartheta \in \SM^\infty(G)$.

Next, let $f,g \in K_2(G)$ be fixed but arbitrary. As above, $\reallywidecheck{f} \mu$ is a finite measure. For simplicity, let $h:= \vartheta*f = \reallywidecheck{\widehat{f}\mu}$, and let $\nu:= h \theta_G$. Then, $\nu$ is a Fourier transformable measure and
\[
\widehat{\nu}= \widehat{f}\mu \,.
\]
Therefore, as $g \in K_2(G)$ we have \cite[Lem. 4.9.26]{MoSt}
\[
\reallywidehat{(g*\nu)\theta_G}=\widehat{f}\widehat{g}\mu=\widehat{f*g}\mu \,.
\]

Now, since $f*g \in K_2(G)$, by Lemma~\ref{lem:properties_sm_1}(b), we have
\[
(\vartheta*(f*g)) \theta_G = \reallywidecheck{\widehat{f*g}\mu}=\reallywidecheck{\reallywidehat{g*\nu}} =(g*\nu) \theta_G= (h*g) \theta_G \,.
\]
This shows that the density functions $\vartheta*(f*g)$ and $(\vartheta*f)*g$ agree $\theta_G$-almost everywhere. As $\vartheta$ is translation bounded, we have $\vartheta*(f*g)$, $(\vartheta*f)*g \in \Cu(G)$ and hence these two functions agree everywhere.

This proves that $\vartheta$ is intertwining.
\end{proof}

Next, we show that for measures, Fourier transformability as a measure and semi-measure coincide.

\begin{lemma}\label{lem FT measure and semi measure}
Let $\mu$ be a measure on $G$. Then, $\mu$ is Fourier transformable as a measure if and only if $\mu$ is Fourier transformable as a semi-measure. Moreover, in this case the two Fourier transforms coincide.
\end{lemma}
\begin{proof}
This follows immediately from the definitions.
\end{proof}

As an immediate consequence of Corollary~\ref{cor FT implies tb and inter} and Lemma~\ref{lem FT measure and semi measure} we get

\begin{corollary}\label{ft imply stb} Let $\mu$ be a measure which is Fourier transformable. Then, $\mu$ is semi-translation bounded.
\end{corollary}

\begin{remark} As already mentioned before, there exists a measure $\mu$ on $\R$ that is Fourier transformable but not translation bounded \cite{ARMA1}.
By Cor.~\ref{ft imply stb} this measure is semi-translation bounded as semi-measure.

\end{remark}

\medskip

Next, let us give a standard classification of Fourier transformability in terms of Fourier--Stieltjes algebra, as well as for Fourier transformable semi-measures with absolutely continuous Fourier transform. Recall here that the Fourier--Stieltjes Algebra $B(G)$ and the Fourier Algebra $A(G)$ of G are defined as
\[
B(G) = \{ \widehat{\mu}\, :\, \mu \in \cM(\widehat{G}) \mbox{ is finite} \}  \qquad \text{ and } \qquad
A(G) = \{ \widehat{f}\, :\, f \in L^1(\widehat{G}) \} \,.
\]

\smallskip
Before moving to the next result let us emphasize that since $A(G) \subseteq B(G) \subseteq \Cu(G)$, if $\vartheta$ is a semi-measure with the property that for all $f \in K_2(G)$ we have $\vartheta*f \in A(G)$, or $\vartheta*f \in B(G)$, respectively, then $\vartheta$ is semi-translation bounded.  This means that all semi-measures in the next Lemma are automatically semi-translation bounded, and hence we can talk about them being intertwining.
\begin{lemma}\label{lem semi FT}
Let $\vartheta$ be a semi-measure. Then,
\begin{itemize}
\item[(a)] $\vartheta$ is Fourier transformable if and only if $\vartheta$ is intertwining and $\vartheta*f \in B(G)$ for all $f \in K_2(G)$. Moreover, in this case, for all
$f \in K_2(G)$ the measure $\widehat{f} \widehat{\mu}$ is finite and
\[
\vartheta*f= \reallywidecheck{\widehat{f}} \,.
\]
\item[(b)] $\vartheta$ is Fourier transformable and $\widehat{\vartheta}$ is an absolutely continuous measure if and only if $\vartheta$ is intertwining and $\vartheta*f \in A(G)$ for all $f \in K_2(G)$.
\end{itemize}
\end{lemma}
\begin{proof}
(a) $\Longrightarrow$ This follows from Lemma~\ref{lem:properties_sm_1}(b) and Corollary~\ref{cor FT implies tb and inter}.

\medskip

\noindent $\Longleftarrow$ For each $f \in K_2(G)$, there exists a finite measure $\nu_f$ on $\widehat{G}$ such that
\[
\vartheta*f= \reallywidecheck{\nu_f} \,.
\]
Then, $\vartheta*f$ is Fourier transformable as measure and \cite[Lem.~4.9.15]{MoSt}
\[
\reallywidehat{ (\vartheta*f)\theta_G} = \nu_f \,.
\]
When we now apply \cite[Lem.~4.9.24]{MoSt} twice, we obtain
\[
\reallywidehat{\widehat{g}\reallywidecheck{\nu_f}}=(\vartheta*f)*g=(\vartheta*g)*f=\reallywidehat{\widehat{f}\reallywidecheck{\nu_g}}
\]
where the intertwining of $\vartheta$ is essential for the middle step.

Therefore, the finite measures $\widehat{g}\reallywidecheck{\nu_f}$ and $\widehat{f}\reallywidecheck{\nu_g}$ coincide. This means
\begin{equation}\label{eq3}
\widehat{g}\reallywidecheck{\nu_f}= \widehat{f}\reallywidecheck{\nu_g} \qquad \text{ for all } f,g \in K_2(G) \,.
\end{equation}

To complete the proof, we now follow the proof of \cite[Thm.~4.5]{BF}. For each $\varphi \in \Cc(G)$ and $g \in K_2(G)$ such that $\varphi(x) \neq 0$ implies $\widehat{g}(x)\neq 0$, we can define
\[
\frac{\varphi}{\hat{g}} (x) =
\begin{cases}
\frac{\varphi(x)}{\widehat{g}(x)} \,, &  \widehat{g}(x) \neq 0 \,,\\
0 \,, &  \mbox{otherwise} \,.
\end{cases}
\]
Note that such a $g$ always exist \cite[Prop.~2.4]{BF} or \cite[Cor.~4.9.12]{MoSt}.
\smallskip
Now, if $g_1,g_2 \in K_2(G)$ are two functions such that $\varphi(x) \neq 0$ implies $\widehat{g_{j}}(x)\neq 0$ for $1 \leqslant j \leqslant 2$, we have
\[
\nu_{g_1}\left(\frac{\varphi}{\widehat{g_1}}\right)
    =\nu_{g_1}\left(\widehat{g_2}\frac{\varphi}{\widehat{g_1}\widehat{g_2}}\right)
     =(\widehat{g_2}\nu_{g_1})\left(\frac{\varphi}{\widehat{g_1}\widehat{g_2}}
       \right)
    \stackrel{\eqref{eq3}}{=}(\widehat{g_1}\nu_{g_2})
      \left(\frac{\varphi}{\widehat{g_1}\widehat{g_2}}\right)
     =\nu_{g_2}\left(\frac{\varphi}{\widehat{g_2}}\right) \,.
\]
This allows us to define a function $\nu :\Cc(\widehat{G}) \to \C$ by
\begin{equation}\label{eq5}
\nu(\varphi):=\nu_{g}\left(\frac{\varphi}{\widehat{g}}\right) \,,
\end{equation}
where $g \in K_2(G)$ is any function such that $\varphi(x) \neq 0$ implies $\widehat{g}(x)\neq 0$. By the above, $\nu$ does not depend on the choice of $g$.

It is easy to see that $\nu$ is linear. Next, let $K \subseteq \widehat{G}$ be any fixed compact set. Pick some $g \in K_2(G)$ such that $\widehat{g} \geqslant 1_K$. Such a $g$ exists by \cite[Prop.~2.4]{BF} or \cite[Cor.~4.9.12]{MoSt}. Set $C_K:= \left| \nu_g \right|(\widehat{G})$, which is finite since $\nu_g$ is a finite measure.

Next, for all $\varphi \in \Cc(\widehat{G})$ with $\supp(\varphi) \subseteq K$, we have
\[
| \nu (g) | = \left| \nu_{g}\left(\frac{\varphi}{\widehat{g}}\right) \right| \leqslant \Big\|\frac{\varphi}{\widehat{g}} \Big\|_\infty | \nu_g |(G) \leqslant C_K \| \varphi \|_\infty  \,.
\]
This shows that $\nu$ is a measure. Moreover, for all $g \in K_2(G)$, and all $\varphi \in \Cc(G)$, we have $\widehat{g}\varphi(x) \neq 0$ implies $\widehat{g}(x)\neq 0$. Therefore, we obtain
\[
(\widehat{g} \nu) (\varphi)= \nu(\widehat{g} \varphi)\stackrel{\eqref{eq5}}{=}\nu_g\left(\frac{\widehat{g} \varphi}{\widehat{g}}\right)=\nu_g(\varphi) \,
\]
and hence
\begin{equation}\label{eq6}
\widehat{g} \nu = \nu_g  \qquad \text{ for all } g \in K_2(G) \,.
\end{equation}
In particular, $\reallywidecheck{g} \nu$ is a finite measure and hence $\reallywidecheck{g} \in L^1(\nu)$ for all $g \in K_2(G)$.

Next, for all $g \in K_2(G)$, we have
\[
\vartheta(g)
    = (\vartheta*g^\dagger)(0)
     = \reallywidecheck{\nu_{g^\dagger}}(0)
     = \int_{\widehat{G}} \dd \nu_{g^\dagger} (t)
     \stackrel{\eqref{eq6}}{=}\int_{\widehat{G}} \dd (\widehat{g^\dagger}\nu)(t)
    =\int_{\widehat{G}} \reallywidecheck{g}(t)\, \dd \nu (t)
    = \nu( \reallywidecheck{g}) \,.
\]
This proves that $\vartheta$ is Fourier transformable and $\widehat{\vartheta}=\nu$.

The last claim of (a) follows from \eqref{eq6}.

\medskip

(b) $\Longrightarrow$ Let $f \in L_{\text{loc}}^1(\widehat{G})$ be such that $\widehat{\vartheta} = f \theta_{\widehat{G}}$, and let $g \in K_2(G)$. Since $\reallywidecheck{g} \in L^1(\widehat{\vartheta})$, we get that $\reallywidecheck{g} \widehat{\vartheta} \in L^1(\widehat{G})$. The claim follows now from Lemma~\ref{lem:properties_sm_1}(b).

\medskip

\noindent $\Longleftarrow$ For all $f \in K_2(G)$, we have $\vartheta*f \in A(G) \subseteq B(G)$. Therefore, by (a), $\vartheta$ is a Fourier transformable semi-measure. Its Fourier transform $\nu =\widehat{\vartheta}$ is then a weakly admissible measure on $\widehat{G}$.

Next, let $f \in K_2(G)$ be arbitrary. Since $\nu$ is weakly admissible, $\widehat{f} \nu$ is a finite measure on $\widehat{G}$, and by Lemma~\ref{lem:properties_sm_1}, we have
\[
\vartheta*f=\reallywidecheck{\widehat{f} \nu} \,.
\]
Since $\vartheta*f \in A(G)$, there exists some $h \in L^1(\widehat{G})$ such that
\[
\vartheta*f=\reallywidecheck{h} \,.
\]
Now, \cite[Lem.~4.9.15]{MoSt} and \cite[Thm.~4.9.13]{MoSt} imply
\[
\widehat{f} \nu = h \theta_{\widehat{G}} \,.
\]
Thus, the measure $\widehat{f} \nu$ is absolutely continuous, for all $f \in K_2(G)$.

Finally, \cite[Prop.~2.4]{BF} or \cite[Cor.~4.9.12]{MoSt} gives that the restriction of $\nu$ to each compact set $K \subseteq \widehat{G}$ is an absolutely continuous measure, and therefore, $\nu$ is absolutely continuous,
which completes the proof.
\end{proof}

Let us emphasize here that, given a semi-measure $\vartheta$, it is in general not easy to check if $\vartheta*f \in A(G)$ and/or $\vartheta*f \in B(G)$ for all $f \in K_2(G)$.
Nevertheless, the result above will be helpful for theoretical applications, allowing us to show that certain classes of semi-measures are Fourier transformable (see Theorem~\ref{bochner thm} below), as well as establishing that Fourier transformable semi-measures have certain properties (see Corollary~\ref{cor 2}).

\medskip

Recall from Proposition~\ref{prop:bijection_ftsm_sam}, that each weakly admissible measure $\nu\in\cM(\widehat{G})$ is the Fourier transform of a semi-measure $\theta_{\nu}$. Moreover, the definition of the Fourier transformability tells us that
\begin{equation*}
\theta_{\nu}(f):=  \nu\big(\reallywidecheck{f\, }\big)\,,
\end{equation*}
for all $f \in K_2(G)$, which gives the explicit description of $\theta_{\nu}$. Let us summarize these observations.

\begin{fact}
Let $\SM_{\cF}(G)$ be the class of Fourier transformable semi-measures. Then, the mapping $\widehat{ \cdot } : \SM_{\cF}(G) \to  \cM^{w}(\widehat{G})$ is a bijection with the inverse $\reallywidecheck{ \cdot } : \cM^{w}(\widehat{G}) \to \SM_{\cF}(G)$ given by
\[
\reallywidecheck{\mu}(f) = \theta_{\mu}(f):= \int_{\widehat{G}} \reallywidecheck{f}( \chi)\ \dd \mu( \chi )  \qquad \text{ for all } f \in K_2(G) \,.
\]
\end{fact}

Let us next briefly discuss when is $\reallywidecheck{\mu}$ a measure. We start with the following preliminary result.

\begin{proposition}  \label{prop:E}
Let $\nu\in\cM(\widehat{G})$ be weakly admissible. Then, the following statements are equivalent.
\begin{itemize}
\item[(i)] $\nu$ is the Fourier transform of a measure $\mu\in\cM(G)$.
\item[(ii)] $\vartheta_{\nu}$ can be uniquely extended to a continuous (in the inductive topology) functional on $\Cc(G)$.
\item[(iii)] For each compact set $K$, there exists some $c_K>0$ such that
\[
\left| \int_{\widehat{G}} \reallywidecheck{f}( \chi) \dd \mu( \chi ) \right| = \left|\vartheta_{\nu}(f) \right| \leqslant c_K \|f \|_\infty
\]
for all $f \in K_2(G)$ with $\supp(f) \subseteq K$.
\end{itemize}
\end{proposition}
\begin{proof}
(i)$\implies$(iii) Let $\mu$ be a measure such that $\nu=\widehat{\mu}$. Let $K\subseteq G$ be a compact subset, and let $f\in K_2(G)$ be such that $\supp(f)\subseteq K$. Since $\mu\in \cM(G)$, there is a constant $c_K>0$ such that $|\mu(f)|\leqslant c_K\, \|f\|_{\infty}$. Consequently, one has
\[
|\vartheta_{\nu}(f)| = \big|\nu\big(\reallywidecheck{f\, }\big)\big| = |\mu(f)|  \leqslant c_K\, \|f\|_{\infty} \,.
\]
This proves (iii).

\medskip

\noindent (iii)$\implies$(ii) This follows from Lemma~\ref{semi is measure}.

\medskip

\noindent (ii)$\implies$(i) By (ii), $\vartheta_{\nu}$ can be uniquely extended to a measure $\mu$. In that case, $\mu$ is Fourier transformable as a semi-measure, and $\widehat{\mu}=
\nu$. The claim follows from Lemma~\ref{lem FT measure and semi measure}.
\end{proof}

Next, given $\nu \in \cM^{w}(G)$, let us discuss when $\theta_{\nu} \in \cM^\infty(G)$.
The following proposition is an immediate consequence of Proposition~\ref{prop tb}, Lemma~\ref{lem:properties_sm_1}, Lemma~\ref{lem 2} and Proposition~\ref{prop:E}.

\begin{proposition}\label{prop FT of tB}
Let $\nu$ be a weakly admissible measure on $\widehat{G}$ and let $U \subseteq G$ be a fixed open pre-compact set. Then, the following statements are equivalent.
\begin{itemize}
\item[(i)] There exists a Fourier transformable measure $\mu \in \cM^\infty(G)$ such that $\widehat{\mu}=\nu$.
\item[(ii)] For each compact set $K \subseteq G$ there exists some constant $C_K>0$ such that
\[
 \bigg|  \int_{\widehat{G}} \chi(t)\, \widehat{f}(\chi)\, \dd \nu (\chi)  \bigg| < C_K
\]
for all $f \in K_2(G)$ with $\supp(f) \subseteq K$ and all $t \in G$.
\item[(iii)] One has
\[
\sup\left\{ \Big|  \int_{\widehat{G}} \chi(t)\, \widehat{f}(\chi)\, \dd \nu (\chi)  \Big|\, :\, f \in  K_2(G) ,\, \supp(f) \subseteq U ,\, \|f \|_\infty \leqslant 1 \,,\, t \in G \right\} < \infty \,.
\]
\end{itemize}\qed
\end{proposition}

This motivates the following definition.

\begin{definition}
Let $U \subseteq G$ be an open pre-compact set. An admissible measure $\nu$ on $\widehat{G}$ is called \emph{$U$-nice} if
\begin{equation*}
\sup\left\{ \bigg|  \int_{\widehat{G}} \chi(t)\, \widehat{f}(\chi)\, \dd \nu (\chi)  \bigg|\, :\, f \in  K_2^{U}(G),\, t \in G \right\} < \infty \,,
\end{equation*}
where
\[
 K_2^{U}(G):= \{  f \in  K_2(G) ,\, \supp(f) \subseteq U , \|f \|_\infty \leqslant 1 \} \,.
\]
\end{definition}

With this definition, Proposition~\ref{prop FT of tB} says that a semi-measure is $U$-nice if and only if it is the Fourier transform of a translation bounded measure. This implies
that the concept of $U$-nice is independent of the choice of $U$ (compare also \cite{SS2}).

\begin{corollary}
Let $U \subseteq G$ be a fixed open pre-compact set, and let $\nu$ be a weakly admissible measure on $\widehat{G}$. Then $\nu$ is $U$-nice if and only if
$\nu$ is $V$-nice for every open pre-compact set $V \subseteq G$. \qed
\end{corollary}

\section{Positive definite semi-measures}\label{sect: pos def}

In the next section, Corollary~\ref{thm:main_decomp}, we will state and prove the analog of \cite[Thm. 4.3]{SS4} for locally compact Abelian groups in a slightly more general version. The reason is that we can apply it to other problems and partially answer an open question raised Argabright and de Lamadrid, see Corollary~\ref{coro:arma}. In order to do so, we need to introduce the following notion.

\begin{definition}\cite[Def.~C6]{LSS}
A semi-measure $\vartheta$ is called \emph{positive definite} if for all $f \in \Cc(G)$ we have
\[
\vartheta(f*\widetilde{f}) \geqslant 0 \,.
\]
\end{definition}

It is clear that any measure is positive definite as a semi-measure if and only if it is positive definite as measure.

\begin{remark}
\begin{itemize}
\item[(a)] Any semi-measure $\vartheta$ induces a linear function $F_\vartheta: \reallywidehat{K_2(G)} \to \C$ via
\[
F_\vartheta(\hat{f}): = \vartheta(f) \,.
\]
We will refer to $F_\vartheta$ as \emph{the functional associated to $\vartheta$}.
\item[(b)] Any linear function $F: \reallywidehat{K_2(G)} \to \C$ induces a semi-measure $\vartheta$ via
\[
\vartheta(f) := F(\widehat{f})\,.
\]
Moreover, in this case $F$ is the functional associated to $\vartheta$.
\item[(c)] $\vartheta$ is positive definite if and only if its associated functional is positive.
\end{itemize}
\end{remark}

This remark emphasizes that (positive definite) semi-measures can be studied via studying (positive) linear functionals on $\reallywidehat{K_2(G)}$.
Later in the paper, given a functional $F: \reallywidehat{K_2(G)} \to \C$ we will say that it is \emph{induced by a measure} if there exists a measure $\nu$ such that
\[
F(g)= \int_{\widehat{G}} g(t)\ \dd \nu(t)
\]
for all
$g \in \reallywidehat{K_2(G)}$. This relation implicitly asks for the integrability of $g$ with respect to $\nu$. Therefore, we get the following trivial result.

\begin{lemma}
\begin{itemize}
  \item[(a)] Let $F: \reallywidehat{K_2(G)} \to \C$ be induced by a measure $\nu$. Then $\nu \in \cM^{W}(\widehat{G})$.
  \item[(b)] A semi-measure $\vartheta$ is Fourier transformable if and only if the functional $F_\vartheta$ associated to $\vartheta$ is induced by a measure $\nu$. Moreover, in this case, we have $\widehat{\vartheta}= \nu$.
\end{itemize}\qed
\end{lemma}

Next, let us emphasize that by Plancherel Theorem we have
\[
\reallywidehat{K_2(G)} \subseteq L^1(\widehat{G}) \,.
\]
Moreover, by the Riemann--Lebesgue Lemma, we have $\reallywidehat{K_2(G)} \subseteq \Cz(\widehat{G})$. In particular, for all $1 \leqslant p \leqslant \infty$ we have
\[
 \reallywidehat{K_2(G)} \subseteq L^1(\hat{G}) \cap \Cz(\widehat{G}) \subseteq L^p(\widehat{G}) \,.
\]
Finally, we also have $ \reallywidehat{K_2(G)} \subseteq A(G) \subseteq B(G)$.

We can now give some examples of positive definite semi-measures. More such examples will be given by Theorem~\ref{bochner thm} and Cor.~\ref{pos FT}.

\begin{example}
Let $1 \leqslant p \leqslant \infty$ and let $F: L^p(\widehat{G}) \to \C$ be positive linear functional. Then,
\[
\vartheta(f)= F(\widehat{f})
\]
is a positive definite semi-measure.
\end{example}

\medskip

Next, we want to give a Bochner type result for semi-measures, similar to the one for measures. The Fourier theory of positive definite measures is well established \cite{BF,ARMA1}. Many of the proofs rely on the following two relations:
\begin{itemize}
  \item{} If $\mu$ is positive definite and $f \in \Cc(G)$, then $\mu*f*\widetilde{f}$ is a continuous positive definite function.
  \item{} If $\mu$ is a (positive definite) measure and $f,g \in \Cc(G)$, then $\left(\mu*f*\widetilde{f}\right)*(g*\widetilde{g}) = \left(\mu*g*\widetilde{g}\right)*(f*\widetilde{f})$.
\end{itemize}

While the second relation trivially holds for all measures, for semi-measures it is equivalent via the polarisation identity \cite[p. 244]{MoSt} to the semi-measure being intertwining.
We show next that for semi-measures, these conditions are equivalent to semi-translation boundedness and intertwining.
This allows us show that semi-translation bounded, intertwining, positive definite semi-measures are Fourier transformable. Corollary~\ref{cor FT implies tb and inter} implies that semi-translation boundedness and intertwining are necessary conditions for Fourier transformability. We will see next that adding intertwining to positive definiteness and semi-translation boundedness gives Fourier transformability. In particular we will get the following fundamental result about positive definite semi-measures (compare \cite[Rem.~C7]{LSS}).

\begin{theorem}[Bochner's theorem for semi-measures]\label{bochner thm}
Let $\vartheta$ be a semi-measure. Then, the following statements are equivalent.
\begin{itemize}
\item[(i)] For all $f \in \Cc(G)$ the function $\vartheta*(f*\widetilde{f})$ is positive definite and continuous, and $\vartheta$ is intertwining.
\item[(ii)] $\vartheta$ is semi-translation bounded, intertwining and for each $f \in \Cc(G)$ there exists a finite measure $\sigma_f$ on $\widehat{G}$ such that
\[
\vartheta*(f*\widetilde{f})= \reallywidecheck{\sigma_f} \,.
\]
\item[(iii)] $\vartheta$ is Fourier transformable and $\widehat{\vartheta}$ is positive.
\item[(iv)] $\vartheta$ is positive definite, semi-translation bounded and intertwining.
\end{itemize}
\end{theorem}
\begin{proof}

(i)$\implies$(ii) Since $\vartheta*(f*\widetilde{f})$ is positive definite and continuous we have $\vartheta*(f*\widetilde{f}) \in \Cu(G)$. The polarisation identity then gives that $\vartheta$ is semi-translation bounded.

\medskip

\noindent (ii) This now follows from Bochner's theorem.

\medskip

\noindent (ii)$\implies$(iii) By (ii), for all $f \in \Cc(G)$ we have $\vartheta*(f*\widetilde{f}) \in B(G)$. The polarisation identity then gives $\vartheta*g \in B(G)$ for all $g \in K_2(G)$. Since $\vartheta$ is intertwining, it is Fourier transformable by Lemma~\ref{lem semi FT}, and, for all $f \in \Cc(G)$ we have
\[
\reallywidecheck{\big|\widehat{f} \big|^2 \widehat{\vartheta}}=\vartheta*(f*\widetilde{f})=\reallywidecheck{\sigma_f} \,.
\]
This immediately gives that $\big|\widehat{f} \big|^2 \widehat{\vartheta}=\sigma_f$ is a positive measure for all $f \in \Cc(G)$. The claim follows.

\medskip

\noindent (iii)$\implies$(iv) By Corollary~\ref{cor FT implies tb and inter}, $\vartheta$ is semi-translation bounded and intertwining.
Let $\mu= \widehat{\vartheta}$. Then, $\mu$ is a positive measure. Therefore, we obtain
\[
\vartheta(f*\widetilde{f}) = \int_{\widehat{G}} | \widehat{f}(\chi) |^2\, \dd \mu(\chi) \geqslant  0
\]
for all $f \in \Cc(G)$. This shows that $\vartheta$ is positive definite.

\medskip

\noindent (iv)$\implies$(i) The function $\vartheta*(f*\widetilde{f})$ is continuous by semi-translation boundedness for all $f \in \Cc(G)$. Following the argument of \cite[Prop.~4.4]{BF}, we show that this is also positive definite.
Since $\vartheta$ is intertwining, we have
\begin{align*}
\int_{G} (\vartheta*(f*\widetilde{f}) )(t)\, (g*\widetilde{g})(t)\ \dd t
    &= \big( (\vartheta*(f*\widetilde{f}) )*(g^\dagger *
       \widetilde{g^\dagger})\big)(0)
     = \left(\vartheta*( (f*\tilde{f})*(g^\dagger*\widetilde{g^\dagger})) \right)
       (0) \\
    &= \left(\vartheta*( (f*g^\dagger)*\widetilde{ (f*g^\dagger) }) \right)(0)
     =\vartheta\left( (f^\dagger*g)*\widetilde{ (f^\dagger*g) } \right)
      \geqslant 0
\end{align*}
for all $g \in \Cc(G)$, with the last inequality following from positive definiteness of $\vartheta$ and $f^\dagger*g \in \Cc(G)$. Therefore, $\vartheta*(f*\tilde{f})$ is positive definite by
\cite[Prop.4.1]{BF}. The claim follows.
\end{proof}

\begin{remark}
Let $\vartheta$ be the semi-measure from Example~\ref{ex4} (2). Then, $\vartheta$ is Fourier transformable and
\[
\widehat{\vartheta}= \lm|_{[0, \infty)} \,.
\]
It follows that $\vartheta$ is positive definite, intertwining and semi-translation bounded. Moreover, since $\widehat{\vartheta} \notin \WAP(\R)$ it follows that $\vartheta$ is not a measure. This gives a simple example of an intertwining, semi-translation bounded measure which is not a measure.

More generally, if $\nu$ is any weakly admissible measure such that $\nu \notin \WAP(\widehat{G})$, then
\[
\vartheta(f) := \int_{\widehat{G}} f(\chi) \dd \nu( \chi) \qquad \text{ for all } f \in K_2(G)
\]
is a semi-translation bounded, intertwining, Fourier transformable semi-measure which is not a measure. Moreover, $\vartheta$ is positive definite if and only if $\nu$ is positive.
\end{remark}

The previous theorem and Proposition~\ref{prop:bijection_ftsm_sam} give the next result.

\begin{corollary}\label{pos FT}
Let $\mu$ be a measure on $\widehat{G}$. Then $\mu$ is positive and weakly admissible if and only if there exists a positive definite semi-measure $\vartheta$, which is semi-translation bounded and intertwining, such that $\widehat{\vartheta}=\mu$. \qed
\end{corollary}

One important open question in the theory of Fourier transform of measures is whether each Fourier transformable measure can be written as a linear combination of positive definite measures \cite{ARMA1}. Recent progress answered this question positively for measures with lattice support \cite{CRS4} and measures with Meyer set support \cite{NS20a}, but the question remains open in general.

Below we show that the equivalent question for semi-measures has a positive answer. In particular, we get that a measure $\mu$ on $G$ is Fourier transformable if and only if
it is a linear combination of (at most four) semi-measures which are positive definite, semi-translation bounded and intertwining.

\begin{proposition}\label{prop arma} Let $\vartheta$ be a semi-measure. Then, $\vartheta$ is Fourier transformable if and only if there exist four semi-measures $\vartheta_j$ which are semi-translation bounded, intertwining and positive definite such that
\[
\vartheta=\vartheta_1-\vartheta_2+\im(\vartheta_3-\vartheta_4) \,.
\]
\end{proposition}
\begin{proof}
$\Longrightarrow$ Since $\vartheta$ is Fourier transformable, there exists a weakly admissible measure $\mu$ on $\widehat{G}$ such that
\[
\widehat{\vartheta}= \mu \,.
\]
As a Fourier transform, $\mu$ is a weakly admissible measure, and so are $\mbox{Re}(\mu)=\frac{\mu+\bar{\mu}}{2}$ and $\mbox{Im}(\mu)=\frac{\mu-\bar{\mu}}{2i}$ by Lemma~\ref{lem:stronglyadmissible_prop}.
Now, using the Hahn decomposition \cite[Thm.~6.14]{rud}
\[
\mbox{Re}(\mu) =\rho_1-\rho_2 \qquad \text{ and } \qquad
\mbox{Im}(\mu) =\rho_3-\rho_4
\]
as difference of orthogonal positive measures, we get
\[
  |\rho_{1,2}| \leqslant |\mbox{Re}(\mu)| \leqslant |\mu|  \qquad \text{ and } \qquad |\rho_{3,4}| \leqslant |\mbox{Im}(\mu)| \leqslant |\mu|  \,.
\]
This gives
\[
\mu=\rho_1-\rho_2+\im(\rho_3-\rho_4)  \,,
\]
where $\rho_j$ is a positive weakly admissible measures for $1\leqslant j\leqslant 4$. By Corollary~\ref{pos FT}, there exist uniformly positive definite semi-measures $\theta_j$ such that
\[
\widehat{\theta_j}=\rho_j  \,.
\]
Thus, one has
\[
\vartheta(f) = \mu(\reallywidecheck{f})=\rho_1(\reallywidecheck{f})-\rho_2(\reallywidecheck{f})+\im\rho_3(\reallywidecheck{f})-\im\rho_4(\reallywidecheck{f})
=\theta_1(f)-\theta_2(f)+\im\theta_3(f)-\im\theta_4(f)
\]
for all $f \in K_2(G)$. This proves the claim.

\medskip

\noindent $\Longleftarrow$ This is a consequence of Theorem~\ref{bochner thm}.
\end{proof}

\begin{corollary} \label{coro:arma}
Let $\mu$ be a measure on $G$. Then $\mu$ is Fourier transformable if and only if $\mu$ is a linear combination of (at most four) positive definite semi-measures which are semi-translation bounded and intertwining.
\end{corollary}
\begin{proof}
This follows from Lemma~\ref{thm:main_decomp} and Lemma~\ref{lem FT measure and semi measure}.
\end{proof}

\begin{example}
For finite pure point measures on $\R$, the result of Corollary~\ref{coro:arma} is immediate. For example consider $\mu=\delta_t$ for some $t\in \R$. Then, we have
\[
\widehat{\mu}=\e^{-2\pi\im t \bullet}\lm = \cos(2\pi t\bullet)_+\lm - \cos(2\pi t\bullet)_-\lm + \im \sin(2\pi t\bullet)_+\lm - \im \sin(2\pi t\bullet)_+\lm \,.
\]
All four measures on the right hand side are periodic (hence Fourier transformable) and translation bounded. Therefore, they are twice Fourier transformable, and thus they are Fourier transforms of some measures.

A similar argument shows that the same holds for all finite measures $\mu$ on $\R$.
\end{example}

\medskip

We complete the section by discussing some natural questions about positive definite semi-measures
Let us first introduce some examples of positive definite semi-measures.

\begin{question}\label{Q2}
Let $\vartheta$ be a positive definite semi-measure on $G$.
\begin{itemize}
\item[(a)]  Is $\vartheta$ semi-translation bounded?
\item[(b)] Assume that $\vartheta$ is semi-translation bounded. Is $\vartheta$ intertwining?
\end{itemize}
\end{question}

By looking at the functional associated to $\vartheta$ we can reformulate these two questions as:

\begin{question}\label{Q3} Let $F: \reallywidehat{K_2(G)} \to \C$ be a positive linear functional. Is $F$ induced by a measure?
\end{question}

Now, if every positive linear functional $F: \reallywidehat{K_2(G)} \to \C$ is induced by a measure, this would imply that each positive definite semi-measure is Fourier transformable, and hence semi-translation bounded and intertwining.

On another hand, if every positive definite semi-measure is semi-translation bounded and intertwining, then by Bochner's theorem it would be Fourier transformable. This would imply that each positive linear functional $F: \reallywidehat{K_2(G)} \to \C$ would be induced by a measure.

\smallskip

Let us emphasize here that on many Banach spaces positivity of linear functionals implies continuity. The following result is standard:

\begin{lemma}\cite{stack}\label{lem stack}
Let $B$ be a Banach space of functions, or equivalence classes of functions from $G$ to $\C$, with the property that
\[
|f| \in B \qquad \text{ and } \qquad \| \left| f \right| \| = \| f \| \,.
\]
for all $f \in B$. If $F : B \to \C$ is a positive linear functional, then $F$ is continuous. \qed
\end{lemma}

The next result follows immediately.

\begin{corollary}
Let $\vartheta$ be a positive definite semi-measure and let  $F_\vartheta: \reallywidehat{K_2(G)} \to \C$ be the associated linear functional.
\begin{itemize}
  \item[(a)] $F_\vartheta$ can be extended to a positive linear functional on $\Cz(\widehat{G})$ if and only if $F_\vartheta$ is induced by a finite positive measure.
  \item[(b)] For $1 \leqslant p < \infty$ and let $q$ be the conjugate of $p$.Then, the functional $F_\vartheta$ can be extended to a positive linear functional on $L^p(\widehat{G})$ if and only if $F_\vartheta$ is induced by an absolutely continuous positive measure $\mu$ with density function $g \in L^q(\widehat{G})$.
\end{itemize}
In particular, in all these cases $\vartheta$ is Fourier transformable, and hence semi-translation bounded and intertwining.
\end{corollary}
\begin{proof}
(a) Follows immediately from Lemma~\ref{lem stack} and the Riesz representation theorem \cite{Rud}.

(b) Follows from Lemma~\ref{lem stack} and the fact that for $1 \leqslant p < \infty$ the dual space of $L^p(G)$ is $L^q(G)$.
\end{proof}

The above results emphasize that any negative answer to Question~\ref{Q2} or Question~\ref{Q3} would yield a positive linear functional $F : K_2(G) \to \C$ which is not induced by a measure, and therefore, cannot be extended to a positive functional neither on $\Cz(G)$ nor $L^p(G)$ for all $1 \leqslant p < \infty$.

\section{Almost periodic semi-measures}\label{sect almost per}

In this section, we introduce the notion of almost periodicity for semi-measures, discuss their Eberlein decomposition and establish the generalized Eberlein decomposition for Fourier transformable semi-measures. The approach to almost periodicity is the standard adaptation of \cite{ARMA,MoSt}, compare \cite{ST}.

\begin{definition}
A semi-measure $\vartheta$ is called \emph{strongly, weakly} or \emph{null weakly} almost periodic if $\vartheta * f \in SAP(G)$, $\vartheta * f \in WAP(G)$ or $\vartheta * f \in WAP_0(G)$, respectively, for all $f \in K_2(G)$. We denote the corresponding spaces of measures by $\sSAP(G)$, $\sWAP(G)$ and $\sWAP_0(G)$, respectively.
\end{definition}

Since $SAP(G)$ and $WAP_0(G)$ are subspaces of $WAP(G) \subseteq \Cu(G)$, the following result holds trivially.

\begin{lemma}
One has the following relations:
\[
\sSAP(G), \sWAP_0(G) \subseteq \sWAP(G) \subseteq \SM^{\infty}(G) \,.
\]
Moreover, $\sSAP(G)$, $\sWAP_0(G)$ and $\sWAP(G)$ are subspaces of $\SM^{\infty}(G)$.
\qed
\end{lemma}

Next, we can talk about the mean and the Fourier--Bohr coefficients of weakly almost periodic semi-measures.

\begin{proposition}\label{prop FB coeff}
\begin{itemize}
\item[(a)] For all $\vartheta \in \sWAP(G)$, $f \in K_2(G)$ and $\chi \in \widehat{G}$, the Fourier--Bohr coefficient
\[
a_{\chi}(\vartheta*f):= \lim_{n\to\infty} \frac{1}{|A_n|} \int_{A_n} \overline{\chi(t)}\, (\vartheta*f)(t)\ \dd t
\]
exists and is independent of the choice of the van Hove sequence $(A_n)_{n\in\N}$.
\item[(b)] Let $\vartheta \in \SM^\infty(G)$ be intertwining, and let $\chi \in \widehat{G}$. If the Fourier--Bohr coefficient $a_{\chi}(\vartheta*f)$ exists for all $f \in K_2(G)$, then there exists a number $a_\chi(\vartheta)$ such that
\[
a_{\chi}(\vartheta*f)=a_{\chi}(\vartheta)\, \widehat{f}(\chi) \,.
\]
for all $f \in K_2(G)$.
\end{itemize}
\end{proposition}
\begin{proof}
(a) This follows from \cite[Prop. 4.5.9]{MoSt}.

\medskip

\noindent (b) Let $f,g \in K_2(G)$, and let $\chi \in \widehat{G}$. Since $\vartheta*f \in WAP(G)$ by \cite[Cor. 3.7]{LSS}, we have
\[
a_{\chi}((\vartheta*f)*g)=\widehat{g}(\chi)\, a_{\chi}(\vartheta*f) \,.
\]
Similarly,
\[
a_{\chi}((\vartheta*g)*f)=\widehat{f}(\chi)\, a_{\chi}(\vartheta*g) \,.
\]
Now, since $\vartheta$ is intertwining, we have $(\vartheta*f)*g=(\vartheta*g)*f$. This shows that
\begin{equation}\label{eq7}
\widehat{g}(\chi)a_{\chi}(\vartheta*f))=\widehat{f}(\chi)a_{\chi}(\vartheta*g)
\end{equation}
for all $f,g \in K_2(G)$ and $\chi \in \widehat{G}$.
Now, let $\chi \in \widehat{G}$ be arbitrary. Fix some $h \in K_2(G)$ such that $\widehat{h}(\chi)=1$. Set
\[
a_{\chi}(\vartheta):= a_{\chi}(\vartheta*h)
\]
Then, Eq.~\eqref{eq7} implies
\[
\widehat{f}(\chi)\, a_{\chi}(\vartheta)=\widehat{f}(\chi) \, a_{\chi}(\vartheta*h)=\widehat{h}(\chi) \, a_{\chi}(\vartheta*f)=a_{\chi}(\vartheta*f)
\]
for all $f \in K_2(G)$, which proves the claim.
\end{proof}

The next result is an immediate consequence.

\begin{corollary}\label{cor 2}
Let $\vartheta \in \sWAP(G)$ be intertwining. Then, for each $\chi\in\widehat{G}$, there exists a number $a_{\chi}(\vartheta)$ such that
\[
a_{\chi}(\vartheta*f)=a_{\chi}(\vartheta) \, \widehat{f}(\chi) \,.
\]
for all $f \in K_2(G)$. \qed
\end{corollary}

\begin{definition}
Let $\vartheta \in \sWAP(G)$ be intertwining. Then, for each $\chi \in \widehat{G}$, the number $a_{\chi}(\vartheta)$ from Corollary~\ref{cor 2} is called the \emph{Fourier--Bohr coefficient of $\vartheta$ at $\chi$}.

The formal sum $\sum_{\chi \in \widehat{G}} a_{\chi}(\vartheta)\, \delta_\chi$ is called the \emph{Fourier--Bohr series of $\vartheta$}.
\end{definition}

We can now give a different characterization for $\sWAP_0(G)$ similar to \cite{ARMA}.

\begin{lemma}
Let $\vartheta \in \sWAP(G)$ be intertwining. Then, $\vartheta \in \sWAP_0(G)$ if and only if
\[
a_\chi(\vartheta)=0 \qquad \text{ for all } \chi \in \widehat{G} \,.
\]
\end{lemma}
\begin{proof}
Since $\vartheta \in \sWAP(G)$, by \cite[Thm.~4.6.12]{MoSt}, we have
\[
M(|\vartheta*f|^2)=\sum_{\chi \in \widehat{G}} | a_\chi(\vartheta*f)|^2
\]
for all $f \in K_2(G)$.
Therefore, by \cite[Remark~4.7.2.]{MoSt}, one has $\vartheta \in \sWAP_0(G)$ if and only if
\[
\sum_{\chi \in \widehat{G}} | a_\chi(\vartheta*f) |^2 =0
\]
for all $f \in K_2(G)$.
The claim follows now from Corollary~\ref{cor 2}.
\end{proof}

For translation bounded measures, almost periodicity as measure and semi-measure coincide. This is an immediate consequence of \cite[Prop.~4.10.3.]{MoSt}.

\begin{lemma}\label{ms vs semi}
We have the following identities of spaces:
\begin{align*}
  \sWAP(G) \cap \cM^\infty(G) &=   \WAP(G)  \,,  \\
  \sSAP(G) \cap \cM^\infty(G) &=   \SAP(G)  \,,  \\
  \sWAP_0(G) \cap \cM^\infty(G) &=   \WAP_0(G) \,.
\end{align*}\qed
\end{lemma}

\begin{remark}
Let $\mu$ be a positive definite measure which is not translation bounded (see again \cite[Prop. 7.1]{ARMA1}). Then, $\mu$ is Fourier transformable as measure, and hence Fourier transformable as semi-measure. Theorem~\ref{genEbe}(a) below then implies that $\mu \in \sWAP(G)$. But, as $\WAP(G) \subseteq \cM^\infty(G)$, we have $\mu \notin \WAP(G)$.
\end{remark}

\medskip

Next, we discuss the completion of these spaces.

\begin{example}
Let $\mu \in \cM(\R)$ be any measure. Define
\[
  \mu_n =\mu|_{[-n,n]} \qquad \text{ and } \qquad
  \nu_n = \mu_n*{\delta_{2n\Z}}
\]
for all $n\in\N$. Then, $\mu_n \in \cM^\infty_0(\R) \subseteq \WAP_0(\R)$ by \cite[Lem.~17]{SS3}.
Moreover, $\nu_n$ is fully periodic and hence $\nu_n \in \SAP(\R)$. In particular,  Lemma~\ref{ms vs semi} implies that $\mu_n \in \sWAP_0(\R)$ and $\nu_n \in \sSAP(\R)$.
By construction, the sequences $(\mu_n)_{n\in\N}$ and  $(\nu_n)_{n\in\N}$ converge vaguely to $\mu$. As $K_2(\R) \subseteq \Cc(\R)$,  $(\mu_n)_{n\in\N}, (\nu_n)_{n\in\N}$ converge in the pointwise topology for semi-measures to $\mu$.
Now, if $\mu \in \cM^\infty(\R) \backslash \WAP(\R)$ it follows from Lemma~\ref{ms vs semi} that $\mu \notin \sWAP(\R)$.

This shows that $\sSAP(\R)$, $\sWAP(\R)$ and $\sWAP_0(\R)$ are not closed in the topology of pointwise convergence.
\end{example}

\begin{remark}
If $G$ is any  non-compact, $\sigma$-compact LCAG, and $(K_n)_{n\in\N}$ is any sequence of compact sets such that
\[
K_{n} \subseteq \mbox{Int}(K_{n+1}) \qquad \text{ and } \qquad
   G = \bigcup_n K_ n\,,
\]
then we have $\mu_n := \mu|_{K_n} \in \cM^\infty_0(G)$ and $\mu_n \to \mu$ pointwise, for all $\mu \in \cM(G)$.
Since $G$ is non-compact,  $\cM(G) \backslash \WAP(G)$ is non-empty. Exactly as above, $\mu_n \in \sWAP_0(G)$,
$(\mu_n)_{n\in\N}$ converges in the pointwise topology for semi-measures to $\mu$ but $\mu \notin \sWAP(G)$.
\end{remark}

We next show that these spaces are closed in a topology similar to the product topology for measures. Let us introduce the following definition first.
Note that each $f \in K_2(G)$ defines a semi-norm $p_f$ on $K_2(G)$ via
\[
\phi_f(\vartheta) = \| f*\vartheta \|_\infty \,.
\]

\begin{definition}
The locally convex topology defined on $\SM^\infty(G)$ by the family of semi-norms $\{ p_f : f \in K_2(G) \}$  is called the \emph{semi-product topology}.
We will denote this topology by $\tau_{\operatorname{sp}}$.
\end{definition}

In this topology, a net $(\vartheta_{\alpha})_{\alpha}$ converges to $\vartheta$ if and only if
\[
\lim_{\alpha} \| \vartheta_{\alpha}*f - \vartheta*f \|_\infty =0
\]
for all $f \in K_2(G)$.
Note that, similarly to \cite{ARMA}, the semi-product topology is simply the topology induced by the embedding
\[
\SM^\infty(G) \hookrightarrow [\Cu(G)]^{K_2(G)}\, \qquad \vartheta \mapsto (\vartheta*f )_{ f \in K_2(G)}\,.
\]

\begin{lemma}
\begin{itemize}
  \item[(a)] $(\SM^\infty(G), \tau_{\operatorname{sp}})$ is complete.
  \item[(b)] The spaces $\sSAP(G)$, $\sWAP(G)$ and $\sWAP_0(G)$ are closed in $(\SM^\infty(G), \tau_{\operatorname{sp}})$. In particular, $(\sSAP(G), \tau_{\operatorname{sp}})$, $(\sWAP(G), \tau_{\operatorname{sp}})$ and $(\sWAP_0(G), \tau_{\operatorname{sp}})$ are complete.
\end{itemize}
\end{lemma}
\begin{proof}
(a) Let $(\mu_\alpha)_\alpha$ be a Cauchy net in $\SM^\infty(G)$. Then, for all $f \in K_2(G)$ the net $(\mu_{\alpha}*f)_\alpha$ is Cauchy in $(\Cu(G), \| \cdot \|_\infty)$ and hence convergent to some $F_f \in \Cu(G)$.

Also note that
\[
\mu_\alpha(f)= (\mu_\alpha*f^\dagger)(0)
\]
for all $f \in K_2(G)$. Since $\mu_{\alpha}*f^\dagger$ is Cauchy in $(\Cu(G), \| \cdot \|_\infty)$, the net $(\mu_{\alpha}(f))_{\alpha}$ is Cauchy in $\C$. Therefore, by Proposition~\ref{sm complete}, there exists a semi-measure $\mu$ such that
\[
\mu(f) = \lim_{\alpha} \mu_\alpha(f) \qquad \text{ for all } f \in K_2(G) \,.
\]
Let $f \in K_2(G)$ and $t \in G$. Then,
\[
F_f(t)= \lim_{\alpha} (\mu_{\alpha}*f)(t)= \lim_{\alpha} \mu_{\alpha}(T_tf^\dagger) =\mu(T_t f)= (\mu*f)(t) \,.
\]
This shows that we have $\mu*f = F_f \in \Cu(G)$ for all $f \in K_2(G)$. Therefore, $\mu \in \SM^\infty(G)$ and
\[
\lim_\alpha \| \mu_\alpha*f -\mu*f \|_\infty = \lim_\alpha \| \mu_\alpha*f -F_f\|_\infty =0
\]
for all $f \in K_2(G)$.
This proves that $\mu_{\alpha}$ converges in $\tau_{sp}$ to $\mu \in \SM^\infty(G)$.

\medskip

\noindent (b) Since $WAP(G)$ and $SAP(G)$ are closed in $(\Cu(G), \| \cdot \|_\infty)$ \cite[Prop.~4.3.4]{MoSt}, it follows immediately that $\sSAP(G)$ and $\sWAP(G)$ are closed in $(\SM^\infty(G), \tau_{\operatorname{sp}})$. Moreover, the inequality $M(|f|) \leqslant \|f \|_\infty$ for all $f \in WAP(G)$ gives that $WAP_0(G)$ is closed in $(WAP(G), \| \cdot \|_\infty)$ which gives that $\sWAP_0(G)$ is closed in $(\SM^\infty(G), \tau_{\operatorname{sp}})$.
\end{proof}

We are now ready to prove the existence of the Eberlein decomposition for semi-measures.

\begin{theorem}\label{t2}
We have
\[
\sWAP(G)= \sSAP(G) \oplus \sWAP_0(G) \,.
\]
In particular, every semi-measure $\vartheta \in \sWAP(G)$ can be written uniquely in the form
\begin{equation}\label{ebe}
\vartheta=\vartheta_s+\vartheta_0 \,,
\end{equation}
with
\[
  (\vartheta_s)*f = (\vartheta*f)_s \qquad \text{ and } \qquad
  (\vartheta_0)*f = (\vartheta*f)_0
\]
for all $f \in K_2(G)$.
\end{theorem}
\begin{proof}
It is easy to see that
\[
\vartheta_s(f) := (\vartheta*f^\dagger)_s(0)  \qquad \text{ and } \qquad
\vartheta_0(f) := (\vartheta*f^\dagger)_0(0)
\]
define linear functionals on $K_2(G)$, and hence semi-measures.
Moreover,  we have
\begin{align*}
(\vartheta_s*f)(0)&=\vartheta_s(T_tf^\dagger) = (\vartheta*(T_{-t}f))_s(0) =\left(T_{-t} (\vartheta*f)\right)_s(0)= (\vartheta*f)_s(t) \,, \\
(\vartheta_0*f)(0)&=\vartheta_0(T_tf^\dagger) = (\vartheta*(T_{-t}f))_0(0) =\left(T_{-t} (\vartheta*f)\right)_0(0) = (\vartheta*f)_0(t)\,,
\end{align*}
for all $t \in G$, which completes the proof.
\end{proof}

\begin{definition}
Let $\vartheta \in \sWAP(G)$. We will refer to the decomposition \eqref{ebe} as the \emph{Eberlein decomposition of} $\vartheta$.
\end{definition}

\begin{remark}
Let $\mu \in \WAP(G)$. Then, Lemma~\ref{ms vs semi} implies that the Eberlein decomposition of $\mu$ as measure and semi-measure coincide.
\end{remark}

Now we can prove the existence of the generalized Eberlein decomposition. We start with some preliminary results.
First, the following lemma is an immediate consequence of Proposition~\ref{prop:bijection_ftsm_sam}.

\begin{lemma}\label{lem proj} Let $T : \cM^w(G) \to \cM^w(G)$ be any function. Then, there exists some function $S: \SM_{T}(G) \to   \SM_{T}(G)$ such that
\[
\reallywidehat{S(\vartheta)} = T( \widehat{\vartheta}) \,.
\]
for all $\vartheta \in  \SM_{T}(G)$. \qed
\end{lemma}

\begin{corollary} \label{thm:main_decomp}
Let $\nu\in\mathcal{M}(\widehat{G})$ be the Fourier transform of a Fourier transformable semi-measure $\vartheta$. Let $\{\nu_j\, |\, 1\leqslant j \leqslant n\} \subseteq \mathcal{M}(\widehat{G})$ such that
\[
\nu = \sum_{j=1}^n \nu_j \qquad \text{ and } \qquad |\nu_j|\leqslant |\nu|\ \text{ for all } 1\leqslant j \leqslant n \,.
\]
Then, there are Fourier transformable semi-measures $\vartheta_1,\ldots,\vartheta_n$ such that
\[
\vartheta = \sum_{j=1}^n \vartheta_j \qquad \text{ and } \qquad
\widehat{\vartheta_j}=\nu_j\,,\ 1\leqslant j \leqslant n \,.
\]
\end{corollary}
\begin{proof}
Since $|\nu_j|\leqslant |\nu|$ for all $1\leqslant j\leqslant n$, this is a direct consequence of Lemma~\ref{lem:wa_sa_prop}(b), Lemma~\ref{lem:stronglyadmissible_prop}(b) and Proposition~\ref{prop:bijection_ftsm_sam}.
\end{proof}

\begin{remark}
Let $U \subseteq G$ be a fixed open pre-compact set. Under the setting of Corollary~\ref{cor 1} we have
\begin{itemize}
  \item[(a)] $\gamma$ is translation bounded and hence $\widehat{\gamma}$ is $U$-nice.
  \item[(b)] $\gamma_{\operatorname{s}}, \gamma_{\operatorname{0a}},\gamma_{\operatorname{0s}}$ are translation bounded measures if and only if $(\widehat{\gamma})_{\operatorname{pp}}, (\widehat{\gamma})_{\operatorname{ac}}, (\widehat{\gamma})_{\operatorname{sc}}$ are $U$-nice.
\end{itemize}
\end{remark}

Next, we want to know if the semi-measures can be measures. The following example shows that we cannot expect this to hold in general.

\begin{example}
Consider $G=\R$. The Lebesgue measure $\nu=\lm$ is the Fourier transform of the (semi-)measure $\vartheta=\delta_0$. Now, if we decompose $\nu$ into
\[
\nu= 1_{\{x\in\R\, |\, x\geqslant 0\}} \lm + 1_{\{x\in\R\, |\, x< 0\}} \lm =: \nu_1+\nu_2 \,,
\]
we can apply Theorem~\ref{thm:main_decomp} and obtain Fourier transformable semi-measures $\vartheta_1$ and $\vartheta_2$ such that
\[
\delta_0=\vartheta_1 + \vartheta_2 \qquad \text{ and } \qquad \widehat{\vartheta_1}= 1_{\{x\in\R\, |\, x\geqslant 0\}} \lm \,,\ \ \ \widehat{\vartheta_2}= 1_{\{x\in\R\, |\, x< 0\}} \lm \,.
\]
However, $\vartheta_1$ and $\vartheta_2$ are not measures, since $\nu_1$ and $\nu_2$ are not weakly almost periodic (whence they can't be Fourier transforms of measures \cite[Thm.~4.11.12]{MoSt}).
\end{example}

Now, we are ready to prove the following result.

\begin{theorem}\label{genEbe}
Let $\vartheta \in \SM_{T}(G)$. Then,
\begin{itemize}
\item[(a)] $\vartheta \in \sWAP(G)$.
\item[(b)]$\vartheta_{\operatorname{s}}, \vartheta_0 \in \SM_{T}(G)$ and
\[
\widehat{\vartheta_{\operatorname{s}}}= (\widehat{\vartheta})_{\operatorname{pp}} \qquad \mbox{ and } \qquad \widehat{\vartheta_0}= (\widehat{\vartheta})_{\operatorname{c}}
\]
\item[(c)] For all $\chi \in \widehat{G}$, we have
\[
\widehat{\vartheta}(\{\chi\}) = a_{\chi}(\vartheta)=a_{\chi}(\vartheta_{\operatorname{s}}) \,.
\]
\item[(d)] There exist semi-measures $\vartheta_{\operatorname{0a}},\vartheta_{\operatorname{0s}} \in \SM_{T}$ such that $\vartheta_0=\vartheta_{\operatorname{0a}}+\vartheta_{\operatorname{0s}}$ and
\[
\widehat{\vartheta_{\operatorname{0s}}}= (\widehat{\vartheta})_{\operatorname{sc}} \qquad \mbox{ and } \qquad \widehat{\vartheta_{\operatorname{0a}}}= (\widehat{\vartheta})_{\operatorname{ac}} \,.
\]
\end{itemize}
\end{theorem}
\begin{proof}
(a) This follows from Lemma~\ref{lem:properties_sm_1} and \cite[Lemma~4.8.6]{MoSt}.

\medskip

\noindent (b) This follows from Theorem~\ref{t2} and \cite[Thm.~4.8.11.]{MoSt}.

\medskip

\noindent (c) Let $f \in K_2(G)$ be such that $\widehat{f}(\chi)=1$. Let $\mu:=\widehat{\vartheta}$.
 Then, Lemma~\ref{lem:properties_sm_1} gives $\reallywidehat{ \reallywidecheck{f} \mu}(t)=(\vartheta*f)(-t)$. Thus, by applying \cite[Lem.~4.8.7]{MoSt} to the finite measure $\reallywidecheck{f} \mu$, we get
\[
\reallywidecheck{f}(\chi) \widehat{\vartheta}(\{\chi\}) = \reallywidecheck{f}(\chi)\, \mu(\{\chi\})= a_{\chi}(\vartheta*f) \,.
\]
Finally, since $\vartheta$ is Fourier transformable, it is intertwining by  Corollary~\ref{cor FT implies tb and inter}, and hence, Proposition~\ref{prop FB coeff} implies
\[
a_{\chi}(\vartheta*f) =\widehat{f}(\chi)\, a_{\chi}(\vartheta) \,.
\]
This shows that
\[
\widehat{\vartheta}(\{\chi\}) = a_{\chi}(\vartheta)
\]
for all Fourier transformable semi-measures $\vartheta$ and all $\chi \in \widehat{G}$
and, by (b),
\[
\widehat{\vartheta}(\{\chi\}) =(\widehat{\vartheta})_{\text{pp}}(\{\chi\}) =a_{\chi}(\vartheta_{\text{s}}) \,,
\]
which proves (c).

\medskip

\noindent (d) By Lemma~\ref{lem:stronglyadmissible_prop} (c), the projections $T_{\text{ac}}(\mu)=\mu_{\text{ac}}$ and $T_{\text{sc}}(\mu)=\mu_{\text{sc}}$ take $\cM^{w}(G)$ into $\cM^{w}(G)$. The claim follows now from Lemma~\ref{lem proj}.
\end{proof}

Let us introduce the following notation:
\begin{align*}
  \sWAP_{0a,T}(G) &= \{ \vartheta \in \SM_{T}(G)\, :\, \widehat{\vartheta} \mbox{ is absolutelly continuous} \} \,, \\
    \sWAP_{0s,T}(G) &= \{ \vartheta \in \SM_{T}(G)\, :\, \widehat{\vartheta} \mbox{ is singular continuous} \} \,,
\end{align*}
Theorem~\ref{genEbe} tells us that we have the decomposition
\[
\SM_{T}(G)= \left( \SM_{T}(G) \cap \sSAP(G) \right) \oplus \underbrace{\sWAP_{0a,T}(G) \oplus   \sWAP_{0s,T}(G) }_{=\sWAP_0(G) \cap \SM_{T}(G)} \,.
\]
Moreover, the Fourier transform induces the following bijections:

\begin{tikzpicture}
  \matrix (m) [matrix of math nodes,row sep=2em,column sep=.4em,minimum width=2em]
  { \SM_{T}(G) &= \left( \SM_{T}(G) \cap \sSAP(G) \right) &\oplus\sWAP_{0s,T}(G) & \oplus \sWAP_{0a,T}(G) \\
    \cM^w(G)&=\cM^w_{pp}(G)&\oplus  \cM^w_{sc}(G)&\oplus \cM^w_{ac}(G) \\};
  \path[<->]
    (m-1-1) edge node [left] {$\mathcal{F}$} (m-2-1)
    (m-1-2) edge node [left] {$\mathcal{F}$} (m-2-2)
    (m-1-3) edge node [left] {$\mathcal{F}$} (m-2-3)
    (m-1-4) edge node [left] {$\mathcal{F}$} (m-2-4);
\end{tikzpicture}

\section{On the components of the generalized Eberlein decomposition}\label{sect comp}

In this section, we review the properties of the components of the generalized Eberlein decomposition. Our approach here is similar to \cite{SS4}.
Let us start with the following simple result.

\begin{lemma} Let $\vartheta \in \sWAP(G)$ and let $(A_n)_{n\in\N}$ be a van Hove sequence. Let $ \vartheta_1, \vartheta_2 \in \SM(G)$ be such that
\[
\vartheta=\vartheta_1+\vartheta_2 \,.
\]
Then $\vartheta_1=\vartheta_{\operatorname{s}}$ and $\vartheta_2= \vartheta_0$ if and only if
\begin{itemize}
\item[(i)] $\vartheta_1 *f \in SAP(G)$,
\item[(ii)] $\lim_{n\to\infty} \frac{1}{A_n} \int_{A_n} | (\vartheta_2*f)(t)|\ \dd t =0$,
\end{itemize}
for all $f \in K_2(G)$.
\end{lemma}
 \begin{proof}
$\Longrightarrow$ Let $f\in K_2(G)$. Note that $f*\vartheta_1 = f *\vartheta_{\text{s}} \in SAP(G)$ and $f*\vartheta_2=f*\vartheta_0\in WAP_0(G)$. But this also implies (ii).

\medskip

\noindent $\Longleftarrow$ Let $f\in K_2(G)$. By (i), $f*\vartheta_1\in SAP(G)$, and hence $f*\vartheta_2\in WAP(G)$. Together with (ii), this gives $f*\vartheta_2\in WAP_0(G)$. Therefore, Theorem~\ref{t2} implies
\[
f*\vartheta_1=(f*\vartheta)_{\text{s}} = f*(\vartheta_{\text{s}}) \qquad \text{ and } \qquad f*\vartheta_2=(f*\vartheta)_0 = f*(\vartheta_0) \,.
\]
Since $f\in K_2(G)$ was arbitrary, the claim follows.
\end{proof}

Next, we prove a Riemann--Lebesgue type lemma for semi-measures.

\begin{lemma}
Let $\vartheta \in \sWAP_{\operatorname{0a,T}}(G)$. Then, $\vartheta*f \in \Cz(G)$ for all $f \in K_2(G)$.
\end{lemma}
\begin{proof}
Let $f\in K_2(G)$. Since $\vartheta \in \sWAP_{\operatorname{0a,T}}(G)$, there is a function $h\in L_{\text{loc}}^1(\widehat{G})$ such that $\widehat{\vartheta}=h\theta_{\widehat{G}}$. By, Lemma~\ref{lem:properties_sm_1}, we have
\begin{equation}  \label{eq:01}
(\vartheta*f)(t)=\reallywidecheck{\widehat{f}h}\, (t) \,.
\end{equation}
Note that
\[
\int_{\widehat{G}} |\widehat{f}(\chi)|\, |h(\chi)|\ \dd \chi
    = \int_{\widehat{G}} |\widehat{f}(\chi)|\ \dd |\widehat{\vartheta}|(\chi)
    < \infty \,,
\]
again by Lemma~\ref{lem:properties_sm_1}. This shows that for all $f \in K_2(G)$ we have $\widehat{f}\, h \in L^1(\widehat{G})$. Consequently, $\vartheta*f\in\Cz(G)$ by Eq.~\eqref{eq:01} and the standard version of the Riemann--Lebesgue lemma.
\end{proof}

Next, we can prove the following result, compare \cite{SS4}.

 \begin{proposition}
Let $\mu\in\mathcal{M}_T(G)$. Then, the following statements are equivalent:
\begin{enumerate}
\item[(i)] $\reallywidecheck{g}*\mu\in L^2(G)$, for all $g\in KL(\widehat{G})$.
\item[(ii)] $g\widehat{\mu}$ is absolutely continuous with $L^2(\widehat{G})$-density, for all $g\in KL(\widehat{G})$.
\item[(iii)] $g\widehat{\mu}$ is absolutely continuous with $L^2(\widehat{G})$-density, for all $g\in KL(\widehat{G})$.
\item[(iv)] $\widehat{\mu}$ is absolutely continuous with $L^2_{loc}(\widehat{G})$-density.
\end{enumerate}
In particular, we have $\mu \in \sWAP_{\operatorname{0a,T}}(G)$.
\end{proposition}
\begin{proof}
First note that $g\widehat{\mu}$ is a finite measure for all $g\in KL(\widehat{G})$ \cite{ARMA1}. Therefore, it is twice Fourier transformable.

\medskip

\noindent (i)$\iff$(ii): We have
\begin{equation} \label{eq:2}
\left\langle \reallywidecheck{g\widehat{\mu}}\, , \, \phi\right\rangle
    = \left\langle g\widehat{\mu}\, , \, \reallywidecheck{\phi}\right\rangle
     = \left\langle \widehat{\mu}\, , \, g\,\reallywidecheck{\phi}\right\rangle
     = \left\langle\mu\, ,\,\reallywidehat{g\reallywidecheck{\phi}}\right\rangle
    =\left\langle \mu\, , \, \widehat{g}*\phi\right\rangle
     = \left\langle (\reallywidecheck{g}*\mu)\theta_G\, , \, \phi\right\rangle
\end{equation}
for all $\phi\in KL(G)$ (note here that $g\,\reallywidecheck{\phi}\in KL(\widehat{G})$ because
\[
\|\reallywidehat{g\reallywidecheck{\phi}}\|_{L^1} = \|\reallywidecheck{g} * \phi\|_{L^1} \leqslant \|\reallywidecheck{g}\|_{L^1}\|\phi\|_{L^1} < \infty
\]
by Young's inequality). Now, \eqref{eq:2} and the dominated convergence theorem imply
\[
\left\langle \reallywidecheck{g\widehat{\mu}}\, , \, f\right\rangle = \left\langle (\reallywidecheck{g}*\mu)\theta_G\, , \, f\right\rangle
\]
for all $f\in\Cc(G)$. The claim now follows.

\medskip

\noindent (ii)$\implies$(iii): follows from $K_2(G) \subseteq KL(G)$.

\medskip

\noindent (iii)$\implies$(iv): Let $K \subseteq G$ be any compact set. Pick some $g\in K_2(\widehat{G})$ with $g\equiv 1$ on $K$. Then, by (ii), there is a function $h_{K} \in L^2(\widehat{G})$ such that $g \widehat{\mu}= h_K \theta_{\widehat{G}}$. In particular, restricting all these measures to $K$ we get
\begin{equation}\label{eq4}
\widehat{\mu}|_{K} = (g\widehat{\mu})|_{K}=(h_K \theta_{\widehat{G}})|_{K}=(h_K)|_{K} \theta_{\widehat{G}} \,.
\end{equation}
This shows that the restriction of $\widehat{\mu}$ to each compact set $K$ is absolutely continuous measure, and hence $\widehat{\mu}$ is an absolutely continuous measure. Let $f$ be its density function. Then \eqref{eq4} gives that $f|_K \theta_{\widehat{G}}=(h_K)|_{K} \theta_{\widehat{G}} $ and hence $f(x)=h_{K}(x)$ for almost all $x \in K$.

Since the restriction of $h_K$ to $K$ is an $L^2$ function and the restrictions of $f$ and $h_K$ to $K$ coincide almost everywhere, we get  $f|_{K} \in L^2(\widehat{G})$.
This gives $f \in L^2_{\text{loc}}(\widehat{G})$ as claimed.

\medskip

\noindent (iv)$\implies$(ii): Let $\widehat{\mu}=h\,\theta_{\widehat{G}}$ with $h\in L^2_{loc}(\widehat{G})$. Then, for all $g \in KL(G)$ the measure $g\widehat{\mu}$ is finite, and  absolutely continuous with density function $gh$. Finally, $gh\in L^2(\widehat{G})$ because
\[
\int_{\widehat{G}} \left| g(\chi)\, h(\chi)\right|^2\,
\dd\theta_{\widehat{G}}(\chi)
    \leqslant \|g\|_{\infty}^2 \int_{\operatorname{supp}(g)}  |h(\chi)|^2\,
      \dd\theta_{\widehat{G}}(\chi) < \infty \,.
\]
\end{proof}

\section{Fourier transform of measures}

According to Proposition~\ref{prop:E}, a weakly admissible measure $\nu\in\mathcal{M}(\widehat{G})$ is the Fourier transform of a measure $m\in\mathcal{M}(G)$ if and only if the corresponding semi-measure $\vartheta_{\nu}$ is continuous. The purpose of this section is to give another characterisation. In order to do so, we need to introduce a special class of sequences of functions, namely positive definite approximate identities.

The next lemma shows that we can pick the approximate identity to be in $K_2(G)\cap P(G)$.

\begin{lemma}\label{L1}
Let $U \subseteq G$ be any open pre-compact neighborhood of $0$. Then, there exists a sequence $(K_n)_{n\in\N}$ in $K_2(G)\cap P(G)$ which is an approximate identity for the convolution in $(\Cu(G), *)$ such that $\operatorname{supp}(K_n) \subseteq U$ for all $n\in \N$.

Moreover, for any such sequence $(K_n)_{n\in\N}$ and all $\mu \in \cM(G)$, the sequence $\big((\mu*K_n)\,\theta_G\big)_{n\in\N}$ converges vaguely to $\mu$.
\end{lemma}
\begin{proof}
Let $V$ be an open neighborhood of $0$ such that $V-V \subseteq U$. Pick an approximate identity $(g_n)_{n\in\N}$ with $\supp(g_n) \subseteq V$. Then, it is easy to verify that
\[
K_n= g_n * \widetilde{g_n}
\]
satisfies the given conditions.

Next, let $\mu \in \cM(G)$ and let $f \in \Cc(G)$. Let $K:= \supp(f)$. Then, $\supp(K_n*f) \subseteq K+ \overline{U}$.
Note that, since $(K_n)_{n\in\N}$ is an approximate identity, we have
\[
\| K_n^\dagger *f -f \|_\infty = \| K_n *f^\dagger -f^\dagger \|_\infty \to 0
\]
Therefore, since $\mu$ is a measure, we get
\begin{align*}
\left| \big((\mu*K_n)\theta_G\big)(f)- \mu(f) \right|
    &=\left| \int_G \int_G  f(t)\, K_n(t-s)\, \dd \mu(s)\, \dd t
       - \int_G f(s)\, \dd \mu(s) \right|  \\
    &=\left| \int_G (  (f*K_n^\dagger)(s) -   f(s))\ \dd \mu(s) \right| \\
    &\leqslant C_{K+\overline{U}}\, \| f*K_n^\dagger - f \|_\infty
\end{align*}
with the last inequality following from the definition of Radon measures.
\end{proof}

\textbf{For the remaining part, let us fix some sequence $(K_n)_{n\in\N}$ as in Lemma~\ref{L1}.}

\medskip

We can now make use of the sequence $(K_n)_{n\in\N}$ to characterise positive measures that are Fourier transforms of other measures.

\begin{theorem}\label{T1}
Let $\nu$ be a positive measure on $\widehat{G}$. Then, there is a measure $\mu$ on $G$ with $\widehat{\mu}=\nu$ if and only if
\begin{enumerate}
\item[(1)] $\nu$ is weakly admissible,
\item[(2)] the set $\{(\widehat{K_n}\,\nu)^{\reallywidecheck{\phantom{..}}}\, : \, n\in\N\}$ is vaguely bounded.
\end{enumerate}
\end{theorem}
\begin{proof}
First, assume that such a measure $\mu$ exists. By definition of the Fourier transform of a measure, $\nu$ is weakly admissible. Moreover, by  \cite[Lem. 4.9.24]{MoSt}, we have $(\mu*K_n)(t)= (\widehat{K_n}\,\nu)^{\reallywidecheck{\phantom{..}}}(t)$, and hence
\[
\mu_n:= (\widehat{K_n}\,\nu)^{\reallywidecheck{\phantom{..}}} = (\mu*K_n)\,\theta_G \xrightarrow{n\to\infty} \mu \,,
\]
by Lemma~\ref{L1}. Therefore, the set $\{\mu_n\, : \, n\in\N\}$ is vaguely bounded \cite[Prop. A.4]{SS}.

\medskip

\noindent On the other hand, assume that properties (1) and (2) hold. Then, $\{\mu_n\, : \, n\in\N\}$ is vaguely compact \cite[Prop.~A.4]{SS}, and $(\mu_n)_{n\in\N}$ has a vague cluster point, say $\mu$. Since $\widehat{K_n}\, \nu$ is a finite positive measure, $\mu_n$ is a positive definite measure. Hence, there is a subsequence $(n_m)_{m\in\N}$ such that $(\mu_{n_m})_{m\in\N}$ converges vaguely to some measure $\mu$. Since $\mu_n$ is positive definite, so is $\mu$. Thus, we obtain
\[
 \lim_{m\to\infty} \widehat{\mu_{n_m}} = \widehat{\mu} \,,
\]
see \cite[Lem.~4.11.10]{MoSt}. Also, we have
\[
\lim_{n\to\infty} \widehat{\mu_{n}} = \lim_{n\to\infty} \widehat{K_n}\,\nu = \nu \,.
\]
Consequently, $\widehat{\mu}$ and $\nu$ coincide.
\end{proof}

\begin{corollary}  \label{cor:FTofmeasure}
Let $\nu$ be a positive measure on $\widehat{G}$. Then, there is a measure $\mu$ on $G$ with $\widehat{\mu}=\nu$ if and only if
\begin{enumerate}
\item[(1)] $\nu$ is weakly admissible,
\item[(2)] for every $g\in\Cc(G)$, there is a constant $c>0$ such that
\[
\big|\nu(\reallywidecheck{g}\, \widehat{K_n})\big| \leqslant c \qquad \text{ for all } n\in\N \,.
\]
\end{enumerate}
\end{corollary}
\begin{proof}
We first show that, for all $g \in \Cc(G)$ and all $n\in\N$, we have
\begin{equation}\label{eq1}
\nu(\reallywidecheck{g}\, \widehat{K_n}) = (\widehat{K_n}\,\nu)^{\reallywidecheck{\phantom{..}}}(g)\,.
\end{equation}
Note that while this looks like the definition of the Fourier transform, but the definition of Fourier transforms only guarantees that \eqref{eq1} holds for all $g \in K_2(G)$, and we need to establish this relation for all $g \in \Cc(G)$.

Note first that $\widehat{K_n}\,\nu$ is a finite measure by weak admissibility, and hence it is twice Fourier transformable \cite[Lem.~4.9.14 and Lem.~4.9.15]{MoSt}.
Then, by \cite[Lem.~4.9.26]{MoSt}, $\reallywidecheck{\widehat{K_n}\,\nu} *(g^{\dagger}\theta_G)$ is a Fourier transformable measure and
\[
\widehat{\reallywidecheck{\widehat{K_n}\,\nu} *(g^{\dagger}\theta_G})= \reallywidecheck{g}\widehat{K_n}\,\nu \,.
\]
Since the right hand side is a finite measure, it is Fourier transformable and
\[
\reallywidecheck{\widehat{K_n}\,\nu} *(g^{\dagger}\theta_G) = (\reallywidecheck{\widehat{K_n}\,\nu} *g^\dagger)\theta_G =I \theta_G \,,
\]
where $I(x)= \reallywidecheck{\reallywidecheck{g}\widehat{K_n}\,\nu}(x)$. Since $I(x)$ and $(\reallywidecheck{\widehat{K_n}\,\nu} *g^{\dagger})(x)$ are continuous functions which are equal as measures, they are equal everywhere. Hence
\[
\nu(\reallywidecheck{g}\, \widehat{K_n})= I(0) = (\reallywidecheck{\widehat{K_n}\,\nu} *g^{\dagger}) (0) = (\widehat{K_n}\,\nu)^{\reallywidecheck{\phantom{..}}}(g) \,.
\]
This proves \eqref{eq1}.

The claim follows now from Theorem~\ref{T1} and \cite[Prop. A.4]{SS}.
\end{proof}

The previous corollary enables us to characterise the Fourier transformable semi-measures that are measures.

\begin{corollary}
A positive definite Fourier transformable semi-measure $\vartheta$ is a measure if and only if, for all $g\in \Cc(G)$, there is a constant $c>0$ such that
\[
|\vartheta(K_n*g)| \leqslant c \,.
\]
\end{corollary}
\begin{proof}
Let $\nu$ be the positive Fourier transform of $\vartheta$. Then, for all $g \in \Cc(G)$ we have
\[
\vartheta(K_n*g)=\nu(\reallywidecheck{g}\, \reallywidecheck{K_n})
\]
The claim follows now from Corollary~\ref{cor:FTofmeasure} with $K_n$ replaced by $K_n^\dagger$.
\end{proof}

\section{The generalized Eberlein decomposition for Fourier transformable measures}

We know that $\gamma_{\text{s}}$ and $\gamma_0$ are not only semi-measures but measures \cite[Thm.~4.10.12]{MoSt}. Therefore, it suffices to show that $\gamma_{\text{0a}}$ can chosen to be a measure because then
\[
\gamma_{\text{0s}} = \gamma - \gamma_{\text{s}} - \gamma_{\text{0a}}
\]
implies that $\gamma_{\text{0s}}$ is a measure as well. Let $(\widehat{\gamma})_{\text{ac}}= h\theta_{\widehat{G}}$ for some $h\in L_{\text{loc}}^1(\widehat{G})$. Now, we can ask the following question: which properties must $h$ satisfy so that $\gamma_{0\text{a}}$ is a measure?

\begin{proposition}
If $h$ satisfies one of the following statements, then $\gamma_{0\text{a}}$ is a measure:
\begin{enumerate}
\item[(a)] $h\in B(\widehat{G})$.
\item[(b)] $h\in L^p(\widehat{G})$ for $1\leqslant p \leqslant 2$.
\item[(c)] $h\in SAP(\widehat{G})$ such that for all compact sets $K \subseteq G$ we have
\[
\sum_{x \in K} \left| a_x(h) \right| < \infty  \,,
\]
where  $a_x(h)= M(x(\chi) h(\chi))$ is the Fourier--Bohr coefficient of $h$.
\item[(d)] $h\in\widehat{L^p(G)}$ for $1\leqslant p \leqslant 2$.
\end{enumerate}

Moreover, in (a),(b), the measure $\gamma_{0a}$ is translation bounded.
\end{proposition}
\begin{proof}
(a) Since $h \in B(\widehat{G})$, there exists a finite measure $\mu$ such that $h=\widehat{\mu}$. Then, by \cite[Lem. 4.9.14]{MoSt}, as measures we have $\widehat{\mu}=h \theta_{\widehat{G}} =(\widehat{\gamma})_{\text{ac}}$.
Then, Lemma~\ref{lem FT measure and semi measure} gives $\gamma_{0a}=\mu$.

\medskip

\noindent (b) Since $(\widehat{\gamma})_{\text{ac}}$ is weakly admissible by Lemma~\ref{lem:stronglyadmissible_prop}(b), it suffices to show that, for every $f\in \Cc(G)$, there is a constant $c>0$ such that
\[
|(\widehat{\gamma})_{\text{ac}}(\reallywidecheck{f}\reallywidecheck{K_n})|\leqslant c \qquad \text{ for all } n\in\N.
\]
This follows from H\"older's inequality, Hausdorff--Young's inequality and Young's inequality for convolutions because
\begin{align*}
|(\widehat{\gamma})_{\text{ac}}(\reallywidecheck{f}\reallywidecheck{K_n})|
    &= \left| \int_{\widehat{G}} \reallywidecheck{f*K_n}(\chi)\, h(\chi)\,
       \dd \theta_{\widehat{G}}(\chi) \right|
       \leqslant \|\reallywidecheck{f*K_n}\|_{L^{p'}}\, \|h\|_{L^p} \\
    &\leqslant \|f*K_n\|_{L^p}\, \|h\|_{L^p}
      \leqslant \|f\|_{L^p}\, \underbrace{\|K_n\|_{L^1}}_{=1}\, \|h\|_{L^p}
     =  \|f\|_{L^p}\, \|h\|_{L^p} \,,
\end{align*}
where $p'$ is the H\"older conjugate of $p$, i.e. $\frac{1}{p}+\frac{1}{p'}=1$.

\medskip

\noindent (c)
By \cite[Thm.~6.5 and Thm.~8.1]{NS20b} there exists a Fourier transformable measure $\mu$ on $G$ such that $\widehat{\mu}= h \theta_{\widehat{G}}$.
Then, Lemma~\ref{lem FT measure and semi measure} gives $\gamma_{0a}=\mu$,
which completes the proof.

\medskip

\noindent (d) Let $f \in L^p(G)$ be such that $\widehat{f}=h$. Then, by \cite[Thm.~2.2]{ARMA1}, $\mu=f \theta_G$ is Fourier transformable as measure and
$\widehat{\mu}=h \theta_{\widehat{G}}$.
Lemma~\ref{lem FT measure and semi measure} again gives $\gamma_{0a}=\mu$,
which completes the proof.
\end{proof}

\section*{Acknowledgements}

This work was supported by the German Research Foundation (DFG), within the
CRC 1283 at Bielefeld University and via research grant 415818660
(TS), and by the Natural Sciences and Engineering Council of Canada
(NSERC), via grant 2020-00038 (NS).

\end{document}